\documentclass[fontsize=11pt,english,reqno]{amsart} 
\usepackage{ucs}
\usepackage{ae}
\usepackage{microtype}
\usepackage[utf8]{inputenc}
\usepackage[T1]{fontenc}
\usepackage[english]{babel}

\usepackage{amsmath}
\usepackage{amssymb}
\usepackage{amstext}
\usepackage{amsthm}
\usepackage{thmtools}

\usepackage{stmaryrd}

\usepackage{enumitem}

\usepackage{geometry}
\usepackage{xcolor}

\newcommand{\EDIT}[1]{{\color{black}#1}}
\newcommand{\EEDIT}[1]{{\color{black}#1}}

\usepackage{mathtools}
\usepackage{slashed}

\usepackage{mathscinet}
\usepackage[numbers]{natbib}

\usepackage{graphicx}
\usepackage{tikz-cd}
\usepackage{caption}

\usepackage{imakeidx}
\makeindex[intoc]

\usepackage{tabularx}

\usepackage[colorlinks=true,hyperindex, linkcolor=magenta, urlcolor=black, pagebackref=false, citecolor=cyan,pdfpagelabels]{hyperref}
\usepackage[capitalize]{cleveref}

\newif\ifcomment
\newcommand{\comment}[1]{\ifcomment#1\fi}

\newlength{\currentparindent}
\makeatletter
\newcommand{\@minipagerestore}{\setlength{\parindent}{\currentparindent}}
\makeatother

\newcommand{\nospacepunct}[1]{\makebox[0pt][l]{#1}}

\usepackage{relsize}
\usepackage[bbgreekl]{mathbbol}
\usepackage{amsfonts}

\DeclareSymbolFontAlphabet{\mathbb}{AMSb}
\DeclareSymbolFontAlphabet{\mathbbl}{bbold}
\newcommand{\Prism}{{\mathlarger{\mathbbl{\Delta}}}}
\newcommand{\prism}{\mathbbl{\Delta}}

\DeclareMathOperator{\Spf}{Spf}
\DeclareMathOperator{\Spec}{Spec}

\DeclareMathOperator{\Pic}{Pic}
\DeclareMathOperator{\Cone}{Cone}

\DeclareMathOperator{\Perf}{Perf}
\DeclareMathOperator{\gr}{gr}
\DeclareMathOperator{\Fil}{Fil}

\DeclareMathOperator{\Vect}{Vect}

\DeclareMathOperator{\fib}{fib}
\DeclareMathOperator{\cofib}{cofib}

\DeclareMathOperator{\Map}{Map}
\DeclareMathOperator{\Rees}{Rees}

\let\Gauge\gauge

\DeclareMathOperator{\Tot}{Tot}

\newcommand{\colim}{\operatornamewithlimits{colim}}

\newcommand{\sHom}{\underline{\mathrm{Hom}}}

\newcommand{\Der}{\mathcal{D}\kern -.5pt er}

\newcommand{\tensorL}{\otimes^\mathbb{L}}

\newcommand\A{\mathbb{A}}
\newcommand\F{\mathbb{F}}
\newcommand\G{\mathbb{G}}

\newcommand\V{\mathbb{V}}
\newcommand\Z{\mathbb{Z}}

\newcommand\T{\mathbb{T}}

\newcommand\D{\mathcal{D}}
\newcommand\DF{\mathcal{DF}}

\let\O\cO

\let\H\calH

\let\T\calT

\newcommand\dHod{\slashed{D}}

\newcommand\dR{\mathrm{dR}}
\newcommand\N{\mathrm{N}}

\newcommand\HT{\mathrm{HT}}

\newcommand\Hod{\mathrm{Hod}}

\newcommand\can{\mathrm{can}}

\newcommand\fl{\mathrm{fl}}
\newcommand\nilp{\mathrm{nilp}}

\let\d\derivative

\newcommand\conj{\mathrm{conj}}

\let\epsilon\varepsilon
\let\phi\varphi
\let\ol\overline

\let\tensor\otimes

\let\cal\mathcal

\newtheorem{thm}{Theorem}[subsection]
\newtheorem{prop}[thm]{Proposition}
\newtheorem{lem}[thm]{Lemma}
\newtheorem{cor}[thm]{Corollary}

\theoremstyle{definition}
\newtheorem{defi}[thm]{Definition}

\newtheorem{rem}[thm]{Remark}

\newtheorem{exx}[thm]{Example}
\newenvironment{ex}
  {\pushQED{\qed}\exx}
  {\popQED\endexx}

\numberwithin{equation}{subsection}

\date{\today}

\makeatletter
\renewcommand{\address}[1]{\gdef\@address{#1}}
\renewcommand{\email}[1]{\gdef\@email{\url{#1}}}
\newcommand{\@endstuff}{\par\vspace{\baselineskip}\noindent\small
\begin{tabular}{@{}l}\scshape\@address\\\textit{E-mail address:} \@email\end{tabular}}
\AtEndDocument{\@endstuff}
\makeatother

\title{A stacky comparison of the Nygaard and Hodge filtrations}
\author{Maximilian Hauck}
\address{Max-Planck-Institut f\"ur Mathematik, Vivatsgasse 7, 53111 Bonn, Germany}
\email{max.hauck01@gmail.com}

\begin{document}

\begin{abstract}
We use the approach to $p$-adic cohomology theories via stacks recently developed by Drinfeld and Bhatt--Lurie to formulate a stacky version of a comparison result between the Nygaard filtration on prismatic cohomology and the Hodge filtration on de Rham cohomology by Bhatt--Lurie and thereby also obtain a generalisation in the case of smooth and proper $p$-adic formal schemes which allows for coefficients in an arbitrary gauge. In the process, we develop a stacky approach to diffracted Hodge cohomology as introduced by Bhatt--Lurie which also captures the conjugate filtration and the Sen operator. In the appendix, we also introduce a stack computing the conjugate filtration on absolute Hodge--Tate cohomology.
\end{abstract}

\maketitle

\tableofcontents

\section{Introduction}

While the idea that one may compute the value of a cohomology theory attached to a scheme $X$ by instead computing the coherent cohomology of a suitably defined stack attached to $X$ goes back, in the case of de Rham cohomology, to work of Simpson in the 1990s, see \cite{Simpson} and \cite{Simpson2}, this approach has only recently found entrance into the field of $p$-adic Hodge theory and has been starting to be fully exploited in the course of the last few years with the formulation of prismatic cohomology in terms of stacks independently developed by Bhatt--Lurie and Drinfeld in \cite{APC}, \cite{PFS}, \cite{FGauges} and \cite{Prismatization}. Roughly speaking, \EDIT{similarly to how one can attach to any $p$-adic formal scheme $X$ its \emph{de Rham stack} $X^\dR$, which has the property that coherent cohomology of the structure sheaf $\O_{X^\dR}$ agrees with the ($p$-completed) de Rham cohomology of $X$ if $X$ is smooth,} they functorially attach a stack $X^\prism$ to any \EDIT{such} $X$ with the feature that coherent cohomology of the structure sheaf $\O_{X^\prism}$ agrees with the (absolute) prismatic cohomology of $X$ \EEDIT{if $X$ is $p$-quasisyntomic in the sense of \cite[Def. C.6]{APC}}; correspondingly, the stack $X^\prism$ is called the \emph{prismatisation} of $X$. Moreover, similarly to how the de Rham stack $X^\dR$ admits a filtered refinement $X^{\dR, +}$ over $\A^1/\G_m$ computing the Hodge filtration, they also introduce a filtered refinement $X^\N$ over $\A^1/\G_m$ of $X^\prism$ computing the (absolute) Nygaard filtration on prismatic cohomology. The aspects of this approach most relevant to this paper are shortly reviewed in Section \ref{sect:stacks}.

The upshot of this picture is twofold: First, various statements about prismatic cohomology and related cohomology theories now admit a ``geometric'' formulation; \EDIT{for example}, the comparison between prismatic cohomology and de Rham cohomology from \cite[Thm. 5.4.2]{APC} can be reinterpreted as saying that, for any smooth $p$-adic formal scheme $X$, there is a functorial isomorphism
\begin{equation*}
(X_{p=0})^\prism\cong X^\dR\;.
\end{equation*}
Second, the stacky formulation immediately furnishes natural categories of coefficients for the respective cohomology theories: these should just be the categories of quasi-coherent complexes, or perhaps perfect complexes, on the corresponding stack. E.g., the category $\D(X^\N)$ of \emph{gauges} on $X$ is a category of coefficients for Nygaard-filtered prismatic cohomology and, for any $E\in\D(X^\N)$, we can define the Nygaard-filtered prismatic cohomology of $X$ with coefficients in $E$ as just the (derived) pushforward of $E$ to $\A^1/\G_m$.

In this paper, we want to further advocate this philosophy by showing that the following comparison theorem between the Hodge and Nygaard filtrations of Bhatt--Lurie, see \cite[Prop.\ 5.5.12]{APC}, admits a ``stacky'' version and thereby also generalising the theorem to gauge coefficients:

\begin{thm}
\label{thm:nygaardhodge-motivation}
Let $X$ be a smooth qcqs $p$-adic formal scheme. Then there is a natural filtered comparison map
\begin{equation*}
\Fil^\bullet_\N R\Gamma_\prism(X)\rightarrow\Fil^\bullet_\Hod R\Gamma_\dR(X)
\end{equation*}
with the property that the induced maps
\begin{equation*}
R\Gamma_\prism(X)/\Fil^i_\N R\Gamma_\prism(X)\rightarrow R\Gamma_\dR(X)/\Fil^i_\Hod R\Gamma_\dR(X)
\end{equation*}
are $p$-isogenies for all $i\geq 0$. If \EEDIT{$0\leq i\leq p$}, they are in fact already isomorphisms integrally.
\end{thm}

\begin{rem}
Note that, in loc.\ cit., the integral version is only stated for $i<p$, but the proof given there in fact also works for the case $i=p$. We also note that the smoothness assumption can be removed via Kan extensions by \cite[Rem.\ 5.5.10]{APC}.
\end{rem}

Namely, the result we are going to prove is the following:

\begin{thm}
\label{thm:nygaardhodge-main}
There is a commutative square of stacks
\begin{equation*}
\begin{tikzcd}
\Z_p^\dR\ar[r]\ar[d, "i_\dR", swap] & \Z_p^{\dR, +}\ar[d, "i_{\dR, +}"] \\
\Z_p^\prism\ar[r, "j_\dR"] & \Z_p^\N
\end{tikzcd}
\end{equation*}
which is an almost pushout up to $p$-isogeny. \EDIT{More precisely,} for any $E\in\Perf(\Z_p^\N)$, it induces a pullback diagram
\begin{equation*}
\begin{tikzcd}
R\Gamma(\Z_p^\N, E)[\frac{1}{p}]\ar[r]\ar[d] & R\Gamma(\Z_p^\prism, j_\dR^*E)[\frac{1}{p}]\ar[d] \\
R\Gamma(\Z_p^{\dR, +}, i_{\dR, +}^*E)[\frac{1}{p}]\ar[r] & R\Gamma(\Z_p^\dR, i_\dR^*E)[\frac{1}{p}]\nospacepunct{\;.}
\end{tikzcd}
\end{equation*}
If the Hodge--Tate weights of $E$ (see \cref{defi:nygaardhodge-htweights}) are all at least \EEDIT{$-p$}, then the statement already holds integrally.
\end{thm}

This will then not only allow us to obtain a new proof of \cref{thm:nygaardhodge-motivation}, but also to establish the following version of \cref{thm:nygaardhodge-motivation} with coefficients in perfect gauges in the case where $X$ is proper:

\begin{cor}
\label{cor:nygaardhodge-coeffs}
Let $X$ be a $p$-adic formal scheme which smooth and proper over $\Spf\Z_p$. For any perfect gauge $E\in\Perf(X^\N)$, there is a natural pullback square
\begin{equation*}
\begin{tikzcd}
\EDIT{\Fil^0_\N R\Gamma_\prism(X, E)[\frac{1}{p}]}\ar[r]\ar[d] & R\Gamma_\prism(X, j_\dR^*E)[\frac{1}{p}]\ar[d] \\
\EDIT{\Fil^0_\Hod R\Gamma_\dR(X, i_{\dR, +}^*E)[\frac{1}{p}]}\ar[r] & R\Gamma_\dR(X, i_\dR^*E)[\frac{1}{p}]\nospacepunct{\;.}
\end{tikzcd}
\end{equation*}
If the Hodge--Tate weights of $E$ (see \cref{defi:nygaardhodge-htweights}) are all at least \EEDIT{$-p$}, then the statement already holds integrally.
\end{cor}

In fact, we not only upgrade the conclusion of \cref{thm:nygaardhodge-motivation} to the language of stacks, but also the proof given in \cite[Prop.\ 5.5.12]{APC}. This uses the theory of \emph{diffracted Hodge cohomology} that was introduced by Bhatt--Lurie in \cite[§4.7]{APC}: \EDIT{For any bounded $p$-adic formal scheme $X$, its diffracted Hodge cohomology $R\Gamma_\dHod(X)$ is a derived $p$-complete complex of $\Z_p$-modules which can be regarded as a deformation of the de Rham cohomology of $X$. Indeed, if $X$ is smooth, after modding out $p$, the two agree up to Frobenius twists, see \cite[Rem. 4.7.18]{APC}; moreover, the complex $R\Gamma_\dHod(X)$ is naturally equipped with an ascending filtration $\Fil^\conj_\bullet R\Gamma_\dHod(X)$ called the \emph{conjugate filtration} with the property that, for smooth $X$, the associated graded pieces of the conjugate filtration on the diffracted Hodge cohomology of $X$ agree with the ones of the Hodge filtration on the de Rham cohomology of $X$ -- more precisely, they both identify with \EEDIT{the} Hodge cohomology \EEDIT{of $X$}, i.e., for any $n\in\Z$, we have
\begin{equation*}
\gr^\conj_n R\Gamma_\dHod(X)\cong R\Gamma(X, \widehat{\Omega}_{X/\Z_p}^n)[-n]\cong \gr^n_\Hod R\Gamma_\dR(X)\;,
\end{equation*}
where $\widehat{\Omega}_{X/\Z_p}$ denotes the $p$-completed cotangent sheaf of $X$ over $\Spf\Z_p$, as usual. Finally, the complex $R\Gamma_\dHod(X)$ is also equipped with an endomorphism $\Theta$ called the \emph{Sen operator}, which can be used to prove an integral refinement of the Deligne--Illusie theorem for smooth $X$ of dimension less than $p$, see \cite[Ex. 4.7.17, Rem. 4.7.18]{APC}. A brief review of the theory of diffracted Hodge cohomology is given in Section \ref{subsect:review-dhod}.} 

Thus, for any bounded $p$-adic formal scheme $X$, we functorially construct stacks $X^\dHod, X^{\dHod, c}$ and $(X^\N)_{t=0}$ which geometrise \EDIT{the diffracted Hodge cohomology $R\Gamma_\dHod(X)$ of $X$ together with its conjugate filtration and the Sen operator.} More precisely, we show:

\begin{thm}
\label{thm:intro-fildhod}
Let $X$ be a bounded $p$-adic formal scheme and assume that $X$ is $p$-quasisyntomic and qcqs.
Then the following are true:
\begin{enumerate}[label=(\roman*)]
\item The pushforward of $\O_{X^\dHod}$ along the map $X^\dHod\rightarrow \Z_p^\dHod\cong\Spf\Z_p$ identifies with $R\Gamma_\dHod(X)$.
\item The pushforward of $\O_{X^{\dHod, c}}$ along the map $X^{\dHod, c}\rightarrow\Z_p^{\dHod, c}\cong\A^1/\G_m$ identifies with $\Fil_\bullet^\conj R\Gamma_\dHod(X)$ under the Rees equivalence.
\item Under the equivalence
\begin{equation*}
\D((\Z_p^\N)_{t=0})\cong \widehat{\D}_{\gr, D-\nilp}(\Z_p\{u, D\}/(Du-uD-1))
\end{equation*}
\EDIT{from \cref{prop:fildhod-zpntzero}}, the underlying graded \EEDIT{$\Z_p\langle u\rangle$-complex} of the pushforward of $\O_{(X^\N)_{t=0}}$ along $(X^\N)_{t=0}\rightarrow (\Z_p^\N)_{t=0}$ identifies with $\Fil_\bullet^\conj R\Gamma_\dHod(X)$ under the Rees equivalence \EEDIT{and, under this identification, the operator 
\begin{equation*}
uD-i: \Fil_i^\conj R\Gamma_\dHod(X)\rightarrow \Fil_i^\conj R\Gamma_\dHod(X)
\end{equation*}
identifies with the Sen operator on $\Fil_i^\conj R\Gamma_\dHod(X)$ for all $i\in\Z$.} 
\end{enumerate}
\end{thm}

While the diffracted Hodge stack $X^\dHod$ has already been introduced in \cite[Constr. 3.8]{PFS}, albeit without explicit mention of the comparison result (i) from \cref{thm:intro-fildhod}, to our current knowledge, neither the filtered refinement $X^{\dHod, c}$ nor the further refinement $(X^\N)_{t=0}$ \EEDIT{appear} anywhere in the literature yet and our study of these might be of independent interest.

In an earlier version of the material in this paper that appeared in the author's master's thesis \cite{Thesis}, the stack $(X^\N)_{t=0}$ was dubbed the ``conjugate-filtered Hodge--Tate stack'' because the author felt like it was somehow related to the conjugate filtration on the (absolute) Hodge--Tate cohomology of $X$, despite the fact that there was not even a map from $(X^\N)_{t=0}$ to $\A^1/\G_m$. However, during the preparation of this manuscript, we found a variant $X^{\HT, c}$ of $(X^\N)_{t=0}$ which actually \emph{does} admit a map to $\A^1/\G_m$ such that the (derived) pushforward of the structure sheaf to $\A^1/\G_m$ computes the conjugate filtration on the Hodge--Tate cohomology of $X$ and we have chosen to include the construction of $X^{\HT, c}$ together with a short proof of the aforementioned fact in the appendix.

\subsection*{Notation and conventions}

We freely make use of the language of $\infty$-categories in the style of Lurie, see \cite{HTT}, and of the theory of derived algebraic geometry as laid out in \cite{DAG}. In particular, we work derived throughout: e.g., all our pullbacks and pushforwards are in the derived sense, i.e.\ when we write $f_*$ for a map $f: \cal{X}\rightarrow\cal{Y}$ of schemes/formal schemes/stacks, we really mean the derived pushforward $Rf_*: \cal{D}(\cal{X})\rightarrow\cal{D}(\cal{Y})$.

All the stacks occurring in this paper are going to be in the fpqc topology on $p$-nilpotent rings and we point out that, for the purposes of this paper, it does not make a difference whether one works in the setting of stacks in groupoids or stacks in $\infty$-groupoids. By a quasi-coherent complex on a stack $\cal{X}$, we mean an object of the derived category $\D(\cal{X})$, which is defined via Kan extension from the affine case as in \cite[§3.2]{DAG}; the same applies to the full subcategories $\Vect(\cal{X})$ and $\Perf(\cal{X})$ of vector bundles and perfect complexes, respectively. 

We frequently make use of the Rees equivalence between the category of quasi-coherent complexes on $\A^1/\G_m$ and the category $\widehat{\DF}(\Z_p)$ of filtered objects in the category of derived $p$-complete $\Z_p$-complexes, see \cite[§2.2.1]{FGauges} or \cite{Moulinos} for an introduction to the Rees equivalence and \cite[§A.1]{DeltaRings}, specifically Ex.\ A.14 in loc.\ cit., for how the restriction to $p$-complete complexes arises. However, beware that our sign convention slightly differs from the one in \cite{FGauges}: for decreasing filtrations, the $\G_m$-action on $\A^1$ we use is given by placing the coordinate $t$ on $\A^1$ in grading degree $-1$ and we indicate this by denoting the quotient by $\A^1_-/\G_m$; consequently, we denote the universal generalised Cartier divisor on $\A^1_-/\G_m$ by $t: \O(1)\rightarrow\O$. This choice has the pleasant effect of removing the change of sign in the passage to the associated graded in \cite[Prop.\ 2.2.6.(3)]{FGauges}. Similarly, in the case of increasing filtrations, we use the notation $\A^1_+/\G_m$ to denote the quotient of $\A^1$ by the $\G_m$-action given by placing the coordinate on $\A^1$, which will be called $u$ this time, in grading degree $1$ and the universal section of $\A^1_+/\G_m$ is denoted $u: \O\rightarrow\O(1)$.

From time to time, we need base change statements for cartesian squares of the form
\begin{equation*}
\begin{tikzcd}
X^?\ar[r]\ar[d] & X^{??}\ar[d] \\
Y^?\ar[r] & Y^{??}
\end{tikzcd}
\end{equation*}
for $?, ??\in\{\dR, \prism, \N, \dots\}$ which are induced by a map $X\rightarrow Y$ of formal schemes. We will usually use these without further justification and refer to \cite[App.\ A]{EtaleCrystalline} for details regarding how to prove such results.

Throughout, $p$ is a fixed prime and $X$ denotes a bounded $p$-adic formal scheme. If $X=\Spf R$ is affine, we also use the notation $R^\dR, R^\prism, \dots$ to denote the stacks $(\Spf R)^\dR, (\Spf R)^\prism, \dots$.

\subsection*{Acknowledgements}

Most of the results in this paper first appeared in my master's thesis and I heartily want to thank my advisor Guido Bosco for his continued support, for many long and fruitful discussions, his constant willingness to answer all of my questions and for lots of helpful comments on an earlier version of the material presented here. This paper was prepared during my time as a PhD student at the Max Planck Institute for Mathematics in Bonn and I wish to thank the institute for its hospitality.

\section{Recollections on stacks and $p$-adic cohomology theories}
\label{sect:stacks}

We briefly remind the reader of the essential input from \cite{Prismatization}, \cite{APC}, \cite{PFS} and \cite{FGauges} needed for our purposes. For a more thorough introduction, we advise the reader to consult these sources directly.

\subsection{De Rham stacks}
\label{subsect:derham}

The idea of the de Rham stack and its filtered refinement goes back to Simpson, see \cite[§5]{Simpson}; for a reference in modern language with more of a focus on the $p$-adic story relevant to us, see \cite[Ch.\ 2]{FGauges}.

\begin{defi}
Over $\A^1_-/\G_m$, the canonical map \EEDIT{$\V(\O(1))^\sharp\rightarrow\G_a^\sharp\rightarrow\G_a$}, where $(-)^\sharp$ denotes the PD hull of the zero section, defines a 1-truncated animated $\G_a$-algebra stack
\begin{equation*}
\EEDIT{
\G_a^{\dR, +}\coloneqq\Cone(\V(\O(1))^\sharp\xrightarrow{\can}\G_a)\;.
}
\end{equation*}
The \emph{Hodge-filtered de Rham stack} $X^{\dR, +}$ of $X$ is the stack $\pi_{X^{\dR, +}}: X^{\dR, +}\rightarrow\A^1_-/\G_m$ defined by
\begin{equation*}
X^{\dR, +}(\Spec S\rightarrow\A^1_-/\G_m)\coloneqq \Map(\Spec\G_a^{\dR, +}(S), X)\;,
\end{equation*}
where the mapping space is computed in derived algebraic geometry. The base change of $X^{\dR, +}$ to $\G_m/\G_m\subseteq \A^1_-/\G_m$ is called the \emph{de Rham stack} of $X$ and denoted $\pi_{X^\dR}: X^\dR\rightarrow \Spf\Z_p$ while the base change $\pi_{X^\Hod}: X^\Hod\rightarrow B\G_m$ of $X^{\dR, +}$ to $B\G_m\subseteq \A^1_-/\G_m$ is called the \emph{Hodge stack} of $X$.
\end{defi}

\begin{rem}
\label{rem:fildrstack-hodstackcoh}
If $X$ is smooth and qcqs over $\Spf\Z_p$, one can describe $X^\Hod$ explicitly as
\begin{equation*}
X^\Hod\cong B_{X\times B\G_m} \V(\T_{X/\Z_p}(1))^\sharp\;,
\end{equation*}
see \cite[Rem.\ 2.5.9]{FGauges}. In particular, by the $\G_m$-equivariant version of \cite[Prop.\ 2.4.5]{FGauges}, this implies that giving a quasi-coherent complex on $X^\Hod$ amounts to specifying a graded quasi-coherent complex $V=\bigoplus_i V_i$ on $X$ equipped with a Higgs field $\Phi: V\rightarrow V\tensor\Omega_{X/\Z_p}^1$ which is locally nilpotent mod $p$ and decreases degree by $1$, i.e.\ the restriction of $\Phi$ to $V_i$ is equipped with a factorisation through $V_{i-1}\tensor\Omega_{X/\Z_p}^1$. The pushforward to $B\G_m$ of such a complex then identifies with the graded complex whose degree $i$ term is given by
\begin{equation*}
R\Gamma(X, \Tot(V_i\xrightarrow{\Phi} V_{i-1}\tensor\Omega_{X/\Z_p}^1\xrightarrow{\Phi} V_{i-2}\tensor\Omega_{X/\Z_p}^2\xrightarrow{\Phi}\dots))\;.
\end{equation*}
\end{rem}

As expected, coherent cohomology on the Hodge-filtered de Rham stack of $X$ computes \EDIT{the} Hodge-filtered ($p$-completed) de Rham cohomology of $X$ in good cases:

\begin{thm}
\label{thm:fildrstack-comparison}
Let $X$ be a \EEDIT{smooth qcqs} $p$-adic formal scheme and consider its Hodge-filtered de Rham stack $\pi_{X^{\dR, +}}: X^{\dR, +}\rightarrow\A^1_-/\G_m$. Then $\H_{\dR, +}(X)\coloneqq \pi_{X^{\dR, +}, *}\O_{X^{\dR, +}}$ identifies with $\Fil^\bullet_{\Hod} R\Gamma_\dR(X)$ in $\widehat{\DF}(\Z_p)$ under the Rees equivalence.
\end{thm}
\begin{proof}
This is \cite[Thm.\ 2.5.6]{FGauges}.
\end{proof}

\begin{rem}
In particular, the above statement also implies that the pushforward $\H_\dR(X)$ of $\O_{X^\dR}$ to $\Spf\Z_p$ identifies with the ($p$-completed) de Rham cohomology $R\Gamma_\dR(X)$ of $X$. Similarly, the pushforward of $\O_{X^\Hod}$ to $B\G_m$ identifies as a graded object with the Hodge cohomology of $X$.
\end{rem}

Motivated by this result, we make the following definition:

\begin{defi}
Let $X$ be a \EEDIT{smooth qcqs} $p$-adic formal scheme. For a quasi-coherent complex $E\in\D(X^\dR)$, we define the \emph{de Rham cohomology} of $X$ with coefficients in $E$ as
\begin{equation*}
R\Gamma_\dR(X, E)\coloneqq \pi_{X^\dR, *}(E)\;.
\end{equation*}
Similarly, the \emph{Hodge-filtered de Rham cohomology} of $X$ with coefficients in a quasi-coherent complex $E\in\D(X^{\dR, +})$ is defined by
\begin{equation*}
\Fil^\bullet_\Hod R\Gamma_\dR(X, E)\coloneqq \pi_{X^{\dR, +}, *}(E)\;.
\end{equation*}
\end{defi}

\subsection{Filtered prismatisation}
\label{subsect:prismsyn}
 
We now describe a stacky approach to prismatic cohomology and the Nygaard filtration. For more details, the reader may consult \cite[Ch.s 4, 5]{FGauges}, but we also want to point out the recent reformulation of the construction of the Nygaard-filtered prismatisation given by Gardner--Madapusi in \cite[§6.4]{GardnerMadapusi}, which avoids the use of the notion of admissible $W$-modules recalled below.

\begin{defi}
\label{defi:prismatisation-cwdiv}
For a $p$-nilpotent ring $S$, a \emph{Cartier--Witt divisor} on $S$ is a generalised Cartier divisor $\alpha: I\rightarrow W(S)$ on $W(S)$ satisfying the following two conditions:
\begin{enumerate}[label=(\roman*)]
\item The \EDIT{ideal generated by the} image of the map $I\xrightarrow{\alpha} W(S)\rightarrow S$ \EDIT{is nilpotent}.
\item The image of the map $I\xrightarrow{\alpha} W(S)\xrightarrow{\delta} W(S)$ generates the unit ideal.
\end{enumerate}
Here, $\delta: W(S)\rightarrow W(S)$ is the usual $\delta$-structure on $W(S)$.
\end{defi}

\begin{defi}
The \emph{prismatisation} $X^\prism$ is the stack over $\Spf\Z_p$ given by assigning to a $p$-nilpotent ring $S$ the groupoid of pairs 
\begin{equation*}
(I\xrightarrow{\alpha} W(S), \Spec \cofib(I\xrightarrow{\alpha} W(S))\rightarrow X)\;,
\end{equation*}
where the first entry is a Cartier--Witt divisor on $S$ and the second entry is a morphism of derived formal schemes.
\end{defi}

\begin{ex}
Any Cartier--Witt divisor $I\xrightarrow{\alpha} W(S)$ yields a generalised Cartier divisor $I\tensor_{W(S)} S\rightarrow S$. Thus, we obtain a morphism of stacks
\begin{equation*}
X^\prism\rightarrow \A^1_-/\G_m\;.
\end{equation*}
The preimage of $B\G_m$ under this map is called the \emph{Hodge--Tate stack} of $X$ and denoted $X^\HT$.
\end{ex}

Similarly to the case of de Rham cohomology, we will use $\H_\prism(X)$ to denote the pushforward of $\O_{X^\prism}$ along $X^\prism\rightarrow\Z_p^\prism$. As before, in good situations, coherent cohomology of the structure sheaf on $X^\prism$ agrees with \EDIT{the} prismatic cohomology of $X$ computed via the absolute prismatic site:

\begin{thm}
\label{thm:prismatisation-comparison}
Let $X$ be a bounded $p$-adic formal scheme and \EEDIT{furthermore} assume that $X$ is $p$-quasisyntomic and qcqs. Then there is a natural isomorphism
\begin{equation*}
R\Gamma(X^\prism, \O_{X^\prism})\cong R\Gamma_\prism(X)\;.
\end{equation*}
\end{thm}
\begin{proof}
Combine \cite[Cor.\ 8.17]{PFS} with \cite[Thm.\ 4.4.30]{APC}.
\end{proof}

\begin{defi}
Let $X$ be a bounded $p$-adic formal scheme which is $p$-quasisyntomic and qcqs. For a quasi-coherent complex $E\in\D(X^\prism)$, we define the \emph{prismatic cohomology} of $X$ with coefficients in $E$ as
\begin{equation*}
R\Gamma_\prism(X, E)\coloneqq R\Gamma(X^\prism, E)\;.
\end{equation*}
\end{defi}

To incorporate the Nygaard filtration into the picture, let $W$ be the ring scheme of $p$-typical Witt vectors, i.e.\ $W\cong\prod_{n\in\mathbb{N}} \Spec\Z[t_1, t_2, \dots]$ as schemes \EDIT{so that the functor of points of $W$ is given by $S\mapsto W(S)$ for any $p$-nilpotent ring $S$ via the Witt components}; if no confusion arises, we will also denote the base change of $W$ to another ring $S$ by the same letter. As usual, we denote the Frobenius and the Verschiebung by $F$ and $V$, respectively. We will consider the following classes of $W$-module schemes:

\begin{defi}
Let $S$ be a $p$-nilpotent ring and $M$ an affine $W$-module scheme over $S$. Then $M$ is called ...
\begin{enumerate}[label=(\roman*)]
\item ... an \emph{invertible $W$-module} if it is fpqc-locally on $R$ isomorphic to $W$ as a $W$-module. 

\item ... an \emph{invertible $F_*W$-module} if it is fpqc-locally on $R$ isomorphic to $F_*W$ as a $W$-module.

\item ... \emph{$\sharp$-invertible} if it is fpqc-locally on $R$ isomorphic to $\G_a^\sharp$ as a $W$-module.
\end{enumerate}
\end{defi}

In part (iii) of the above definition, $\G_a^\sharp$ is regarded as a $W$-module in the following manner:

\begin{lem}
\label{lem:admissible-fundamentalseq}
There is a unique isomorphism $W[F]\cong\G_a^\sharp$ lifting the composition $W[F]\subseteq W\rightarrow\G_a$, where $W[F]$ denotes the kernel of the Frobenius. In fact, there is a short exact sequence of $W$-module sheaves
\begin{equation*}
\begin{tikzcd}
0\ar[r] & \G_a^\sharp\ar[r] & W\ar[r] & F_*W\ar[r] & 0\nospacepunct{\;.}
\end{tikzcd}
\end{equation*}
\end{lem}
\begin{proof}
See \cite[Lem.\ 2.6.1, Rem.\ 2.6.2]{FGauges}.
\end{proof}

There is actually a very straightfoward classification of the three classes of $W$-modules introduced above:

\begin{lem}
\label{lem:admissible-classification}
Let $S$ be a $p$-nilpotent ring. 
\begin{enumerate}[label=(\roman*)]
\item The groupoid of invertible $W$-modules is equivalent to the Picard groupoid of $W(S)$ via the construction $L\mapsto L\tensor_{W(S)} W$ for $L\in\Pic(W(S))$, \EEDIT{where the sheaf $L\tensor_{W(S)} W$ is defined by
\begin{equation*}
(L\tensor_{W(S)} W)(T)\coloneqq L\tensor_{W(S)} W(T)
\end{equation*}
for any $S$-algebra $T$.}

\item The groupoid of invertible $F_*W$-modules is equivalent to the groupoid of invertible $W$-modules via the construction $M\mapsto F_*M$ for invertible $W$-modules $N$.

\item The category of $\sharp$-invertible modules is equivalent to the category of invertible $S$-modules via the construction $L\mapsto \V(L)^\sharp$ for $L\in\Pic S$.
\end{enumerate}
\end{lem}
\begin{proof}
See \cite[Constr. 5.2.2]{FGauges}.
\end{proof}

Note that, for $L\in\Pic(W(S))$, the sheaf $L\tensor_{W(S)} W$ is representable by the scheme $\Spec(A\tensor \Z[t_1, t_2, \dots])$, where $A$ is the ring of global sections of the vector bundle $\V(L)$, and that, for $L\in\Pic S$, the scheme $\V(L)^\sharp$ is actually affine since it is a $\G_a^\sharp$-torsor over $\Spec S$ and $\G_a^\sharp$ is affine.

\begin{ex}
\label{ex:admissible-cwdivtoinvertible}
Any Cartier--Witt divisor $I\xrightarrow{\alpha} W(S)$ gives rise to a morphism of invertible $W$-module schemes $M\xrightarrow{d} W$ over $S$ by \cref{lem:admissible-classification}.
\end{ex}

\begin{defi}
An \emph{admissible} $W$-module is an affine $W$-module scheme $M$ which (as a sheaf) can be written as an extension of an invertible $F_*W$-module by a $\sharp$-invertible $W$-module.
\end{defi}

\begin{rem} One can show that if a $W$-module $M$ is admissible, then the exact sequence witnessing its admissibility is unique, see \cite[Rem.\ 5.2.5]{FGauges}; moreover, the association sending $M$ to this exact sequence is functorial. We call this the \emph{admissible sequence} of $M$ and denote it by
\begin{equation*}
\begin{tikzcd}
0\ar[r] & \V(L_M)^\sharp\ar[r] & M\ar[r] & F_*M'\ar[r] & 0\nospacepunct{\;.}
\end{tikzcd}
\end{equation*}
In particular, this shows that being admissible is an fpqc-local property and hence any invertible $W$-module is admissible by \cref{lem:admissible-fundamentalseq}.
\end{rem}

\begin{defi}
A \emph{filtered Cartier--Witt divisor} on a $p$-nilpotent ring $S$ consists of an admissible $W$-module $M$ and a map $d: M\rightarrow W$ such that the induced map $F_*M'\rightarrow F_*W$ of associated invertible $F_*W$-modules comes from a Cartier--Witt divisor on $S$ via the construction of \cref{ex:admissible-cwdivtoinvertible}. We denote the groupoid of filtered Cartier--Witt divisors on $S$ by $\Z_p^\N(S)$ and this defines a stack $\Z_p^\N$ on $p$-nilpotent rings.
\end{defi}

\begin{ex}
\label{ex:filteredprism-drmap}
Any generalised Cartier divisor $L\xrightarrow{t} S$ yields a filtered Cartier--Witt divisor
\begin{equation*}
\V(L)^\sharp\oplus F_*W\xrightarrow{(t^\sharp, V)} W\;.
\end{equation*}
Thus, we obtain a morphism of stacks
\begin{equation*}
i_{\dR, +}: \A^1_-/\G_m\rightarrow\Z_p^\N
\end{equation*}
called the \emph{Hodge-filtered de Rham map}. We warn the reader that, despite the notation, the map $i_{\dR, +}$ is \emph{not} a closed immersion.
\end{ex}

Given a filtered Cartier--Witt divisor $M\xrightarrow{d} W$, we obtain an induced map of admissible sequences
\begin{equation*}
\begin{tikzcd}
0\ar[r] & \V(L_M)^\sharp\ar[d, "\sharp(d)"]\ar[r] & M\ar[d, "d"]\ar[r] & F_*M'\ar[d, "F_*(d')"]\ar[r] & 0 \\
0\ar[r] & \G_a^\sharp\ar[r] & W\ar[r] & F_*W\ar[r] & 0
\end{tikzcd}
\end{equation*}
by functoriality.

\begin{defi}
Fix a filtered Cartier--Witt divisor $M\xrightarrow{d} W$ on a $p$-nilpotent ring $S$. The map $\sharp(d)$ uniquely has the form $t(d)^\sharp$ for some $t(d): L_M\rightarrow S$ by \cref{lem:admissible-classification}. Hence, sending $d$ to $t(d)$ defines a map of stacks
\begin{equation*}
t: \Z_p^\N\rightarrow\A^1_-/\G_m
\end{equation*}
called the \emph{Rees map}.
\end{defi}

\begin{rem}
\label{rem:filprism-jdr}
One can show that the preimage of $\G_m/\G_m\subseteq\A^1_-/\G_m$ under the Rees map identifies with $\Z_p^\prism$, see \cite[Constr. 5.3.5]{FGauges}, and hence we obtain an open immersion
\begin{equation*}
j_\dR: \Z_p^\prism\rightarrow\Z_p^\N\;.
\end{equation*}
\end{rem}

\begin{defi}
The construction carrying a filtered Cartier--Witt divisor $M\xrightarrow{d} W$ on a $p$-nilpotent ring $S$ to 
$R\Gamma_\fl(\Spec S, \Cone(M\xrightarrow{d} W))$ yields a $1$-truncated animated $W$-algebra stack $\G_a^\N\rightarrow\Z_p^\N$ by \cite[Prop.s 5.3.8, 5.3.9]{FGauges}. The \emph{Nygaard-filtered prismatisation} $X^\N$ of $X$ is the stack $\pi_{X^\N}: X^\N\rightarrow\Z_p^\N$ defined by
\begin{equation*}
X^\N(\Spec S\rightarrow\Z_p^\N)\coloneqq \Map(\Spec\G_a^\N(S), X)\;,
\end{equation*}
where the mapping space is computed in derived algebraic geometry.
\end{defi}

One can check that the open immersion $j_\dR: \Z_p^\prism\rightarrow\Z_p^\N$ induces an open immersion $j_\dR: X^\prism\rightarrow X^\N$. Moreover, observe that, over $\A^1_-/\G_m$, there is an equivalence
\begin{equation}
\label{eq:filprism-drmap}
\EEDIT{
\Cone(\V(\O(1))^\sharp\oplus F_*W\xrightarrow{(t^\sharp, V)} W)\cong\Cone(\V(\O(1))^\sharp\xrightarrow{t^\sharp} \G_a)
}
\end{equation}
and hence, we obtain a map of stacks
\begin{equation*}
i_{\dR, +}: X^{\dR, +}\rightarrow X^\N
\end{equation*}
living over the map $\A^1_-/\G_m\rightarrow\Z_p^\N$ from \cref{ex:filteredprism-drmap}, which we also call the \emph{Hodge-filtered de Rham map}. Restricting $i_{\dR, +}$ to $\G_m/\G_m\subseteq\A^1_-/\G_m$ and $B\G_m\subseteq\A^1_-/\G_m$, respectively, yields morphisms of stacks
\begin{equation*}
i_\dR: X^\dR\rightarrow X^\N\;, \hspace{0.5cm} i_\Hod: X^\Hod\rightarrow X^\N
\end{equation*}
called the \emph{de Rham map} and the \emph{Hodge map}, respectively.

In good situations, coherent cohomology on the \EEDIT{Nygaard-filtered prismatisation} of $X$ computes \EDIT{the} Nygaard-filtered prismatic cohomology of $X$ (here, we use the definition of the Nygaard filtration on absolute prismatic cohomology from \cite[Def. 5.5.3]{APC}):

\begin{thm}
\label{thm:filprism-comparison}
Let $X$ be a bounded $p$-adic formal scheme. Assume that $X$ is $p$-quasisyntomic and qcqs. Then $t_{X, *}\O_{X^\N}$ identifies with $\Fil^\bullet_\N R\Gamma_\prism(X)$ in $\widehat{\DF}(\Z_p)$ under the Rees equivalence. Here, the \emph{Rees map} $t_X$ is the composition of $\pi_{X^\N}: X^\N\rightarrow\Z_p^\N$ with $t: \Z_p^\N\rightarrow\A^1_-/\G_m$.
\end{thm}
\begin{proof}
This follows from \cite[Cor.\ 5.5.11, Rem.\ 5.5.18]{FGauges} and \cite[Cor.\ 5.5.21]{APC} via quasisyntomic descent.
\end{proof}

The category $\D(X^\N)$ provides a sensible notion of coefficients for Nygaard-filtered prismatic cohomology.

\begin{defi}
The category $\D(X^\N)$ is called the category of \emph{gauges} on $X$ and denoted $\Gauge_\prism(X)$. If $X=\Spf R$ is affine, we also write $\Gauge_\prism(R)$ in place of $\Gauge_\prism(\Spf R)$.
\end{defi}

\begin{ex}
By pushforward along the morphism $\pi_{X^\N}: X^\N\rightarrow\Z_p^\N$, we obtain a gauge $\H_\N(X)\coloneqq \pi_{X^\N, *}\O_{X^\N}$, which we call the \emph{structure gauge} of $X$. 
\end{ex}

\begin{defi}
\label{defi:nygaardhodge-htweights}
Let $X$ be a bounded $p$-adic formal scheme and $E\in\D(X^\N)$ a gauge on $X$. Then the pullback of $E$ along
\begin{equation*}
\EEDIT{
\begin{tikzcd}[ampersand replacement=\&, column sep=large]
X\times B\G_m\ar[r] \& X^{\Hod}\ar[r, "i_\Hod"] \& X^\N\;,
\end{tikzcd}
}
\end{equation*}
where the first map is induced by the canonical map \EEDIT{$\G_a\rightarrow B\V(\O(1))^\sharp\oplus \G_a=\G_a^\Hod$ of stacks over $B\G_m$,} identifies with a graded quasi-coherent complex $M^\bullet$ on $X$ and the set of integers $i$ such that \EEDIT{the $i$-th graded piece $M^i$} is nonzero is called the set of \emph{Hodge--Tate weights} of $E$.
\end{defi}

\begin{defi}
Let $X$ be a bounded $p$-adic formal scheme which is $p$-quasisyntomic and qcqs. For a gauge $E$ on $X$, we define the \emph{Nygaard-filtered prismatic cohomology} of $X$ with coefficients in $E$ as
\begin{equation*}
\EDIT{\Fil^\bullet_\N R\Gamma_\prism(X, E)\coloneqq t_{X, *}(E)\;.}
\end{equation*}
\end{defi}

\section{The conjugate-filtered diffracted Hodge stack}
\label{sect:conjdhod}

\EDIT{In this section, we will develop a stacky formulation of diffracted Hodge cohomology in a similar spirit to what has been done in Section \ref{sect:stacks} for de Rham cohomology and prismatic cohomology.} Namely, we will introduce stacks $\pi_{X^\dHod}: X^\dHod\rightarrow\Spf\Z_p$ and $\pi_{X^\dHod, c}: X^{\dHod, c}\rightarrow\A^1_+/\G_m$ attached to any bounded $p$-adic formal scheme $X$ whose coherent cohomology computes \EDIT{the} diffracted Hodge cohomology of $X$ together with its conjugate filtration as introduced in \cite[§4.7]{APC} in good cases. Finally, we will also show that one can naturally incorporate the Sen operator $\Theta$ on diffracted Hodge cohomology into this picture in a way that encodes the divisibility properties of the Sen operator with respect to the conjugate filtration; \EDIT{namely, for any $n\in\Z$, the endomorphism 
\begin{equation*}
\Theta+n: \Fil^\conj_n R\Gamma_\dHod(X)\rightarrow\Fil^\conj_n R\Gamma_\dHod(X)
\end{equation*}
factors uniquely through $\Fil^\conj_{n-1} R\Gamma_\dHod(X)$, see \cite[Rem.\ 4.9.10]{APC}, and this will be reflected in the stacky picture.}

\subsection{Recollections on diffracted Hodge cohomology}
\label{subsect:review-dhod}

We start by briefly recalling the definition of diffracted Hodge cohomology from \cite[§4.7]{APC}. To do this, first recall that, for a prism $(A, I)$, we may extend relative Hodge--Tate cohomology as a functor $R\mapsto R\Gamma_{\ol{\Prism}}(R/A)$ on $\ol{A}$-algebras \EEDIT{$R$} to all animated $\ol{A}$-algebras via left Kan extension as in \cite[Constr. 4.1.3]{APC}; here, we write $\ol{A}\coloneqq A/I$. Moreover, the conjugate filtration $\Fil_\bullet^\conj R\Gamma_{\ol{\Prism}}(R/A)$ \EEDIT{is an increasing filtration} on $R\Gamma_{\ol{\Prism}}(R/A)$ \EEDIT{and can be} obtained by left Kan extension from the full subcategory of smooth $\ol{A}$-algebras, where it is just defined to be the canonical filtration (i.e.\ the Postnikov filtration) on $R\Gamma_{\ol{\Prism}}(R/A)$, see \cite[Rem.\ 4.1.7]{APC}.

Now recall that
\begin{equation*}
\D(\Z_p^\HT)\cong \lim_{(A, I)\in (\Z_p)_\prism} \widehat{\D}(A/I)\;,
\end{equation*}
\EDIT{see \cite[Rem.\ 3.5.3]{APC}.} Thus, for any ring $R$, the association $(A, I)\mapsto \Fil_\bullet^\conj R\Gamma_{\ol{\Prism}}(\ol{A}\tensorL R/A)$ defines a filtered quasi-coherent complex $\Fil_\bullet^\conj\H_{\ol{\prism}}(R)$ on $\Z_p^\HT$, whose underlying unfiltered object we denote by $\H_{\ol{\prism}}(R)$. \EDIT{Essentially} by definition, we have $\H_{\ol{\prism}}(R)=\H_\prism(R)|_{\Z_p^\HT}$ and the Hodge--Tate comparison from \cite[Thm.\ 6.3]{Prisms} implies that there is an isomorphism
\begin{equation}
\label{eq:fildhod-htcomparison1}
\gr_n^\conj\H_{\ol{\prism}}(R)\cong L\widehat{\Omega}_R^n\tensor\O_{\Z_p^\HT}\{-n\}[-n]
\end{equation}
for any $n$.

Now recall from \cite[Thm.\ 3.4.13]{APC} that $\Z_p^\HT\cong B\G_m^\sharp$ and that the category $\D(B\G_m^\sharp)$ admits the following straightfoward linear algebraic description:

\begin{prop}
\label{prop:fildhod-bgmsharp}
There is an equivalence of categories
\begin{equation*}
\D(B\G_m^\sharp)\cong \widehat{\D}_{(\Theta^p-\Theta)-\nilp}(\Z_p[\Theta])
\end{equation*}
induced by pullback along the quotient map $\Spf\Z_p\rightarrow B\G_m^\sharp$. Here, we demand the \EDIT{local} nilpotence of $\Theta^p-\Theta$ only mod $p$; \EEDIT{in other words, the right-hand side denotes the full subcategory of all $M\in\widehat{\D}(\Z_p[\Theta])$ for which $\Theta^p-\Theta$ acts locally nilpotently on $H^i(M\tensorL_{\Z_p} \F_p)$ for all $i\in\Z$.}
\end{prop}
\begin{proof}
See \cite[Thm.\ 3.5.8]{APC}.
\end{proof}

Thus, we may identify $\Fil^\conj_\bullet \H_{\ol{\prism}}(R)$ with a filtered $p$-complete $\Z_p$-complex $\Fil^\conj_\bullet\widehat{\Omega}_R^\dHod$ equipped with an operator $\Theta: \Fil^\conj_\bullet\widehat{\Omega}_R^\dHod\rightarrow\Fil^\conj_\bullet\widehat{\Omega}_R^\dHod$ such that $\Theta^p-\Theta$ is \EEDIT{locally} nilpotent \EEDIT{in cohomology} mod $p$, \EDIT{which we call the Sen operator}; the underlying unfiltered object of $\Fil^\conj_\bullet\widehat{\Omega}_R^\dHod$ is denoted $\widehat{\Omega}_R^\dHod$. As the $n$-th power of the Breuil-Kisin twist identifies with the $\Z_p$-module $\Z_p$ with $\Theta$ acting by multiplication by $n$ under the equivalence from \cref{prop:fildhod-bgmsharp}, by virtue of (\ref{eq:fildhod-htcomparison1}), we have isomorphisms
\begin{equation}
\label{eq:fildhod-htcomparison2}
\gr_n^\conj\widehat{\Omega}_R^\dHod\cong L\widehat{\Omega}^n_R[-n]
\end{equation}
for all $n$ \EDIT{and the Sen operator acts by multiplication by $-n$ on $\gr_n^\conj\widehat{\Omega}_R^\dHod$}. 

\begin{rem}
\label{rem:fildhod-siftedcolims}
\EDIT{Using (\ref{eq:fildhod-htcomparison2}) and the fact that} the functor $R\mapsto L\widehat{\Omega}^n_R[-n]$ from commutative rings to \EDIT{$\widehat{\D}(\Z_p)$} commutes with sifted colimits, we conclude by induction that the functor $R\mapsto \Fil^\conj_\bullet\widehat{\Omega}^\dHod_R$ commutes with sifted colimits as well \EDIT{(note that $\Fil^\conj_\bullet\widehat{\Omega}^\dHod_R=0$ for $\bullet<0$).}
\end{rem}

Finally, as the construction $R\mapsto\Fil^\conj_\bullet\widehat{\Omega}_R^\dHod$ is compatible with $p$-complete étale localisation by \cite[Prop.\ 4.7.11]{APC}, the constructions above glue in the following sense: For any bounded $p$-adic formal scheme $X$, we obtain a filtered quasi-coherent complex $\Fil^\conj_\bullet\widehat{\Omega}_X^\dHod$ on $X$ with the property that
\begin{equation*}
\Fil^\conj_\bullet\widehat{\Omega}_X^\dHod|_{\Spf R}=\Fil^\conj_\bullet\widehat{\Omega}_R^\dHod
\end{equation*}
for any affine open $\Spf R\rightarrow X$. Again, we denote the underlying unfiltered object by $\widehat{\Omega}_X^\dHod$. Moreover, also the corresponding Sen operators glue to give a Sen operator $\Theta$ on $\Fil^\conj_\bullet\widehat{\Omega}_X^\dHod$. This finally leads to the following definition:

\begin{defi}
\label{defi:fildhod-defdhod}
Let $X$ be a bounded $p$-adic formal scheme. The \emph{diffracted Hodge cohomology} of $X$ is defined as
\begin{equation*}
R\Gamma_\dHod(X)\coloneqq R\Gamma(X, \widehat{\Omega}_X^\dHod)\;.
\end{equation*}
It is equipped with a \emph{conjugate filtration} given by 
\begin{equation*}
\Fil^\conj_\bullet R\Gamma_\dHod(X)\coloneqq R\Gamma(X, \Fil^\conj_\bullet\widehat{\Omega}_X^\dHod)
\end{equation*}
and a \emph{Sen operator} $\Theta: \Fil^\conj_\bullet R\Gamma_\dHod(X)\rightarrow\Fil^\conj_\bullet R\Gamma_\dHod(X)$ induced by the Sen operator on $\Fil^\conj_\bullet\widehat{\Omega}_X^\dHod$.
\end{defi}

\subsection{The diffracted Hodge stack}

We now examine the stack $X^\dHod$ as introduced by Bhatt--Lurie in \cite[Constr. 3.8]{PFS} \EDIT{and relate it to \cref{defi:fildhod-defdhod}}. For the rest of this section, let $X$ be a bounded $p$-adic formal scheme.

\begin{defi}
\label{defi:fildhod-dhod}
For a bounded $p$-adic formal scheme $X$, its \emph{diffracted Hodge stack} $X^\dHod$ is defined as the pullback
\begin{equation*}
\begin{tikzcd}
X^\dHod\ar[r, "i_\dHod"]\ar[d, "\pi_{X^\dHod}", swap] & X^\HT\ar[d] \\
\Spf\Z_p\ar[r, "\eta"] & \Z_p^\HT\nospacepunct{\;,}
\end{tikzcd}
\end{equation*}
where the map $\eta: \Spf\Z_p\rightarrow\Z_p^\HT$ identifies with the quotient map under the isomorphism $\Z_p^\HT\cong B\G_m^\sharp$.
\end{defi}

\begin{rem}
Despite the notation, the map $i_\dHod$ is generally \emph{not} a closed immersion. Moreover, recalling from \cite[Thm.\ 3.4.13]{APC} that the map $\eta: \Spf\Z_p\rightarrow \Z_p^\HT$ corresponds to assigning to any $p$-nilpotent ring $S$ the Cartier--Witt divisor $W(S)\xrightarrow{V(1)} W(S)$, we immediately see that
\begin{equation*}
X^\dHod(S)=\Map(\Spec(W(S)/V(1)), X)\;,
\end{equation*}
where the mapping space is computed in derived algebraic geometry and the quotient $W(S)/V(1)$ is to be taken in the derived sense. 
\end{rem}

The stack $X^\dHod$ is related to the diffracted Hodge cohomology of $X$ as follows:

\begin{prop}
\label{prop:fildhod-comparison}
Let $X$ be a bounded $p$-adic formal scheme and assume that $X$ is $p$-quasisyntomic and qcqs. Then $\pi_{X^\dHod, *}\O_{X^\dHod}$ identifies with $R\Gamma_\dHod(X)$ in $\widehat{\D}(\Z_p)$.
\end{prop}
\begin{proof}
As the constructions $X\mapsto X^\dHod$ and $X\mapsto \widehat{\Omega}_X^\dHod$ are compatible with Zariski localisation by \cite[Rem.\ 3.9]{PFS} and \cite[Prop.\ 4.7.11]{APC}, we may assume that $X=\Spf R$ is affine. As $\pi_{R^\prism, *}\O_{R^\prism}=\H_\prism(R)$ by \cite[Cor.s 8.13, 8.17]{PFS} and $\H_{\ol{\prism}}(R)=\H_\prism(R)|_{\Z_p^\HT}$, we conclude that $\pi_{R^\HT, *}\O_{R^\HT}=\H_{\ol{\prism}}(R)$. Now the result follows from base change for the cartesian square
\begin{equation*}
\EDIT{
\raisebox{\depth}{
\begin{tikzcd}[baseline={(current bounding box.center)}, ampersand replacement=\&]
R^\dHod\ar[r, "i_{\dHod}"]\ar[d, "\pi_{R^\dHod}", swap] \& R^\HT\ar[d, "\pi_{R^\HT}"] \\
\Z_p^\dHod\ar[r] \& \Z_p^\HT\nospacepunct{\;.}
\end{tikzcd}
}
}
\qedhere
\end{equation*}
\end{proof}

In the sequel, we will also need a more concrete description of $X^\dHod$ in the case where $X=\Spf R$ for $R$ quasiregular semiperfectoid. 

\begin{rem}
\label{rem:fildhod-degzero}
\EEDIT{For $R$ quasiregular semiperfectoid,} the complex $\widehat{\Omega}_R^\dHod$ is concentrated in degree zero. Indeed, the conjugate filtration is complete and its graded pieces $L\widehat{\Omega}^n_R[-n]$ are all concentrated in degree zero, \EEDIT{see \cite[Rem.\ 4.21]{THHandPAdicHodgeTheory}}, \EEDIT{hence $\Fil^\conj_n \widehat{\Omega}_R^\dHod$ is concentrated in degree zero for any $n\in\Z$ and we conclude that the same is true for $\widehat{\Omega}_R^\dHod$.}
\end{rem}

\begin{prop}
\label{prop:fildhod-dhodqrsp}
Assume that $X=\Spf R$ for a quasiregular semiperfectoid ring $R$. Then there is an isomorphism $R^\dHod\cong \Spf\widehat{\Omega}^\dHod_R$.
\end{prop}
\begin{proof}
Let $(\Z_p[\![\tilde{p}]\!], (\tilde{p}))$ denote the prism from \cite[Not. 3.8.9]{APC}. Then the map $\eta$ identifies with the base change of $\rho_{\Z_p[\![\tilde{p}]\!]}: \Spf\Z_p[\![\tilde{p}]\!]\rightarrow\Z_p^\prism$ to $\Z_p^\HT$ by \cite[Prop.\ 3.8.12]{APC} and hence $R^\dHod$ identifies with the relative \EDIT{Hodge--Tate stack} of $R$ with respect to $(\Z_p[\![\tilde{p}]\!], (\tilde{p}))$ as defined in \cite[Var.\ 5.1]{PFS}. \EDIT{Thus, the stack $R^\dHod$ is actually an affine formal scheme by \cite[Cor.\ 7.18]{PFS} and the claim follows from \cref{prop:fildhod-comparison}.}
\end{proof}

\subsection{Incorporating the conjugate filtration}

We now construct the conjugate-filtered diffracted Hodge stack $X^{\dHod, c}$. For this, we first make some preliminary remarks. Namely, recall that, by the proof of \cite[Prop.\ 5.3.7]{FGauges}, there is an fpqc cover $\A^1_+/\G_m\rightarrow (\Z_p^\N)_{t=0}$. This cover is given by a filtered Cartier--Witt divisor $M\xrightarrow{d} W$ on $\A^1_+/\G_m$ constructed by virtue of the commutative diagram
\begin{equation*}
\EEDIT{
\begin{tikzcd}[ampersand replacement=\&]
0\ar[r] \& \G_a^\sharp\ar[r]\ar[d, "u^\sharp", swap] \& W\ar[r]\ar[d]\ar[dd, "V(1)" {yshift=20pt, xshift=5pt}, bend right=45, swap] \& F_*W\ar[r]\ar[d, equals] \& 0 \\
0\ar[r] \& \V(\O(1))^\sharp\ar[r]\ar[d, "0", swap] \& M\ar[r]\ar[d, "d", swap] \& F_*W\ar[r]\ar[d, "p"] \& 0 \\
0\ar[r] \& \G_a^\sharp\ar[r] \& W\ar[r] \& F_*W\ar[r] \& 0\nospacepunct{\;;}
\end{tikzcd}
}
\end{equation*}
here, $M$ is defined as the pushout of the upper left square and the maps out of $M$ are induced by the other maps in the diagram using the universal property.

\begin{defi}
\label{defi:fildhod-dhodconj}
In the situation above, we obtain a 1-truncated animated $W$-algebra stack
\begin{equation*}
\G_a^{\dHod, c}\coloneqq\Cone(M\xrightarrow{d} W)
\end{equation*}
over $\A^1_+/\G_m$. The \emph{conjugate-filtered diffracted Hodge stack} $X^{\dHod, c}$ of $X$ is the stack over $\A^1_+/\G_m$ defined by
\begin{equation*}
X^{\dHod, c}(\Spec S\rightarrow\A^1_+/\G_m)\coloneqq \Map(\Spec\G_a^{\dHod, c}(S), X)\;,
\end{equation*}
where the mapping space is computed in derived algebraic geometry. If $X=\Spf R$ is affine, we also write $R^{\dHod, c}$ in place of $X^{\dHod, c}$.
\end{defi}

\begin{rem}
\label{rem:fildhod-pullback}
It is immediate from the definition above that one may alternatively describe $X^{\dHod, c}$ as a pullback
\begin{equation*}
\begin{tikzcd}
X^{\dHod, c}\ar[r, "i_{\dHod, c}"]\ar[d, "\pi_{X^{\dHod, c}}", swap] & (X^\N)_{t=0}\ar[d] \\
\A^1_+/\G_m\ar[r] & (\Z_p^\N)_{t=0}\nospacepunct{\;,}
\end{tikzcd}
\end{equation*}
\EEDIT{where the bottom map is given by the filtered Cartier--Witt divisor on $\A^1_+/\G_m$ constructed above.} Also note that the preimage of $\G_m/\G_m\subseteq\A^1_+/\G_m$ under $\pi_{X^{\dHod, c}}$ recovers $X^\dHod$: Indeed, in this case, $u$ is an isomorphism and hence $M\xrightarrow{d} W$ identifies with $W\xrightarrow{V(1)} W$. Moreover, the preimage of $B\G_m\subseteq\A^1_+/\G_m$ under $\pi_{X^{\dHod, c}}$ recovers $X^\Hod$: This follows from the fact that $u=0$ in this case and hence $M\xrightarrow{d} W$ identifies with \EEDIT{$\V(\O(1))^\sharp\oplus F_*W\xrightarrow{(0, V)} W$}.
\end{rem}

The aim of the rest of this section is to prove the following relation between the stack $X^{\dHod, c}$ and the conjugate filtration on the diffracted Hodge cohomology of $X$:

\begin{thm}
\label{thm:fildhod-comparisonfiltered}
Let $X$ be a bounded $p$-adic formal scheme and assume that $X$ is $p$-quasisyntomic and qcqs. Then $\pi_{X^{\dHod, c}, *}\O_{X^{\dHod, c}}$ identifies with $\Fil_\bullet^\conj R\Gamma_\dHod(X)$ under the Rees equivalence. 
\end{thm}

To prove this result, \EDIT{we first have to show the corresponding assertion in the special case where the $p$-adic formal scheme $X$ is smooth:}

\begin{lem}
\label{lem:fildhod-comparisonfilteredsmooth}
Let $X$ be a \EEDIT{smooth qcqs} $p$-adic formal scheme. Then $\pi_{X^{\dHod, c}, *}\O_{X^{\dHod, c}}$ identifies with $\Fil_\bullet^\conj R\Gamma_\dHod(X)$ under the Rees equivalence. 
\end{lem}
\begin{proof}
Similarly to the proof of \cref{prop:fildhod-comparison}, we may reduce to the case where $X=\Spf R$ is affine. As the underlying unfiltered object of $\pi_{X^{\dHod, c}, *}\O_{X^{\dHod, c}}$ identifies with $\widehat{\Omega}_R^\dHod$ by \cref{prop:fildhod-comparison} and \EDIT{$\Fil^\conj_\bullet\widehat{\Omega}_R^\dHod$ is the canonical filtration on $\widehat{\Omega}_R^\dHod$ in this case, all we have to show is that $\pi_{X^{\dHod, c}, *}\O_{X^{\dHod, c}}$ identifies with the canonical filtration on its underlying unfiltered object. Hence, by \cite[Ex.\ 2.2.3]{FGauges},} it suffices to show that $\pi_{X^{\dHod, c}, *}\O_{X^{\dHod, c}}$ is derived $u$-complete and that its $n$-th graded piece is concentrated in degree $n$ for any $n$. For the claim about the graded pieces, note that the pullback of $\pi_{X^{\dHod, c}, *}\O_{X^{\dHod, c}}$ to $B\G_m$ identifies with the Hodge cohomology $\bigoplus_n L\widehat{\Omega}^n_R[-n]$ of $X$ by virtue of \cref{thm:fildrstack-comparison} and \cref{rem:fildhod-pullback}. To prove $u$-completeness, we use the strategy from the proof of \cite[Thm.\ 2.7.9]{FGauges}: First observe that $\pi_0(\G_a^{\dHod, c})\cong\G_a$ and \EEDIT{$\pi_1(\G_a^{\dHod, c})\cong\V(\O(1))^\sharp$}. Thus, the projection $\G_a^{\dHod, c}\rightarrow\G_a$ is a square-zero extension of the target by \EEDIT{$B\V(\O(1))^\sharp$}
over $\A^1_+/\G_m$ and hence the induced map
\begin{equation*}
\nu: X^{\dHod, c}\rightarrow X\times\A^1_+/\G_m
\end{equation*}
of stacks over $\A^1_+/\G_m$ is a \EEDIT{$\V(\T_{X/\Z_p}(1))^\sharp$-gerbe}, \EDIT{where $\T_{X/\Z_p}$ denotes the tangent sheaf of $X$ over $\Z_p$: Indeed, for any $p$-nilpotent test ring $S$ equipped with a map to $\A^1_+/\G_m$, the animated algebras $\G_a^{\dHod, c}(S)$ and $\G_a(S)$ are also $p$-nilpotent and hence we find some $n\geq 0$ such that
\begin{equation*}
X^{\dHod, c}(S)=\Map(\Spec \G_a^{\dHod, c}(S), X)=\Map(\Spec \G_a^{\dHod, c}(S), X_{p^n=0})
\end{equation*}
and similarly $(X\times\A^1_+/\G_m)(S)=\Map(\Spec \G_a(S), X_{p^n=0})$; now the claim follows by derived deformation theory.}\footnote{More precisely, we are using the following assertion in derived algebraic geometry: Let $X$ be a finite type $\Z_p$-scheme and $A'\rightarrow A$ a square-zero extension of animated $\Z_p$-algebras with fibre $N\in\D^{\leq 0}(A)$. Then the fibre of the map $X(A')\rightarrow X(A)$ over a point $\eta: \Spec A\rightarrow X$ is a torsor for $\operatorname{Der}_{\Z_p}(\O_X, \eta_*N)\cong\Map_A(\eta^*L_{X/\Z_p}, N)$; \EDIT{the proof is similar to \cite[Thm.\ 5.1.13]{FlatPurity}.} Note that, if furthermore $N=L[1]\in\D^{\leq -1}(A)$ and $X$ is smooth, we have $\Map_A(\eta^*L_{X/\Z_p}, N)\cong B(\eta^*\T_{X/\Z_p}\tensor_A L)$.} Finally, as the relative cohomology sheaves $H^i(\nu_*\O_{X^\dHod, c})$ of such a gerbe are independent of the gerbe structure by the proof of \cite[Cor.\ 2.7.2.(1)]{FGauges} and it is enough to check that each of them is $u$-complete, we may replace $\nu$ with the trivial gerbe \EEDIT{$\nu': B_{X\times\A^1_+/\G_m}\V(\T_{X/\Z_p}(1))^\sharp\rightarrow X\times\A^1_+/\G_m$}. \EEDIT{Now the conclusion follows as the arguments from \cite[Prop.\ 2.4.5, Rem.\ 2.4.6]{FGauges} show that the pushforward of the structure sheaf along $\nu'$ is given by the quasi-coherent complex 
\begin{equation*}
\EEDIT{
\Tot(\O_X\rightarrow \widehat{\Omega}^1_{X/\Z_p}(-1)\rightarrow \widehat{\Omega}^2_{X/\Z_p}(-2)\rightarrow\dots)
}
\end{equation*}
on $X\times \A^1_+/\G_m$, which is clearly $u$-complete.}
\end{proof}

\EDIT{From the smooth case of \cref{thm:fildhod-comparisonfiltered} treated in the previous lemma, we can now infer an explicit description of the conjugate-filtered diffracted Hodge stack of a quasiregular semiperfectoid ring.}

\begin{prop}
\label{prop:fildhod-dhodconjqrsp}
Assume that $X=\Spf R$ for a quasiregular semiperfectoid ring $R$. Then there is an isomorphism 
\begin{equation*}
R^{\dHod, c}\cong \Spf\Rees(\Fil^\conj_\bullet\widehat{\Omega}^\dHod_R)/\G_m\;.
\end{equation*}
\end{prop}
\begin{proof}
We mimic the proof of \cite[Thm.\ 5.5.10]{FGauges}. As in loc.\ cit., we will work with derived stacks, i.e.\ stacks on animated rings, since we will have to ponder derived pullbacks, which are more natural in the world of derived stacks; however, it will turn out in the end that all stacks we consider were classical to begin with. For this, we extend all the stacks we have introduced so far to the category of animated rings by left Kan extension as in \cite[Rem.\ 5.5.13]{FGauges}. To begin the proof, observe that we can define the conjugate filtration on diffracted Hodge cohomology of $p$-complete animated rings via left Kan extension from the full subcategory of smooth $\Z_p$-algebras since it commutes with sifted colimits \EDIT{by \cref{rem:fildhod-siftedcolims}}. Thus, by \cref{lem:fildhod-comparisonfilteredsmooth}, there is a natural map
\begin{equation*}
\Rees(\Fil^\conj_\bullet\widehat{\Omega}^\dHod_R)\rightarrow \pi_{R^{\dHod, c}, *}\O_{R^{\dHod, c}}
\end{equation*}
of commutative algebra objects in $\D(\A^1_+/\G_m)$. \EEDIT{Recalling that the left-hand side is concentrated in degree zero by \cref{rem:fildhod-degzero}, we see that, by adjunction, the above defines a morphism}
\begin{equation*}
R^{\dHod, c}\rightarrow \Spf\Rees(\Fil^\conj_\bullet\widehat{\Omega}^\dHod_R)/\G_m
\end{equation*}
over $\A^1_+/\G_m$. To check that this is an isomorphism, we may pull back to $\G_m/\G_m\subseteq\A^1_+/\G_m$ and $B\G_m\subseteq\A^1_+/\G_m$ as these form a stratification of $\A^1_+/\G_m$.\footnote{\EEDIT{Here we are using that, given any animated ring $A$ and an element $f\in\pi_0(A)$, a map $\tau: M\rightarrow N$ in $\D(A)$ is an isomorphism if and only if both $\tau[1/f]$ and $\tau\tensorL_A A/f$ are isomorphisms.}} However, after pulling back to $\G_m/\G_m$, the above morphism becomes
\begin{equation*}
R^\dHod\rightarrow \Spf\widehat{\Omega}^\dHod_R\;,
\end{equation*}
which can be verified to agree with the isomorphism from \cref{prop:fildhod-dhodqrsp}, and after pulling back to $B\G_m$, we obtain
\begin{equation*}
R^\Hod\rightarrow \Spf\left(\bigoplus_n L\widehat{\Omega}^n_R[-n]\right)/\G_m\;,
\end{equation*}
which is also an isomorphism, see the proof of \cite[Thm.\ 5.5.10]{FGauges}.
\end{proof}

\EDIT{Finally, we have made all the necessary preparations to prove \cref{thm:fildhod-comparisonfiltered} as announced in the beginning of the section.

\begin{proof}[Proof of \cref{thm:fildhod-comparisonfiltered}]
Since the Nygaard-filtered prismatisation takes quasisyntomic covers to fpqc covers by \cite[Cor.\ 6.12.8]{GardnerMadapusi}, the same is true for the conjugate-filtered diffracted Hodge stack by base change. Thus, by quasisyntomic descent for $\Fil^\conj_\bullet R\Gamma_\dHod(X)$, see \cite[Rem.\ 4.7.9]{APC}, we are reduced to the case where $X=\Spf R$ for a quasiregular semiperfectoid ring $R$. However, in this case, the claim follows from \cref{prop:fildhod-dhodconjqrsp}.
\end{proof}
}

\begin{rem}
Observe that the mod $p$ reduction $(X^{\dHod, c})_{p=0}$ of the conjugate-filtered diffracted Hodge stack agrees with the conjugate-filtered de Rham stack of $X_{p=0}$ defined in \cite[§2.7]{FGauges} up to a Frobenius twist. This reflects the fact that, for smooth $X$, the isomorphism
\begin{equation*}
\widehat{\Omega}_{X^{(1)}}^\dHod\tensor_{\Z_p} \F_p\cong F_*\widehat{\Omega}_{X_{p=0}/\F_p}^\bullet\;,
\end{equation*}
coming from the crystalline comparison theorem for prismatic cohomology, see \cite[Rem.\ 4.7.18]{FGauges}, is compatible with the respective conjugate filtrations on either side. Here, $X^{(1)}\coloneqq \phi^* X$ denotes the pullback of $X$ along the Frobenius of $\Z_p$ and $F: X_{p=0}\rightarrow X_{p=0}^{(1)}$ is the relative Frobenius for $X_{p=0}$ over $\F_p$.
\end{rem}

\subsection{The Sen operator}

To begin our analysis \EDIT{of the stack $(X^\N)_{t=0}$ and to relate it to the Sen operator on the conjugate-filtered diffracted Hodge cohomology of $X$}, we first study the stack \EDIT{$(\Z_p^\N)_{t=0}$} and its category of quasi-coherent complexes.

\begin{prop}
\label{prop:fildhod-zpntzero}
There is a natural identification
\begin{equation*}
(\Z_p^\N)_{t=0}\cong \G_{a, +}^\dR/\G_m=\G_{a, +}/(\G_a^\sharp\rtimes\G_m)\;,
\end{equation*}
where the subscript $(-)_+$ indicates that $\G_m$ acts on $\G_a$ by placing the coordinate $u$ in grading degree $1$, as before.
\end{prop}
\begin{proof}
See \cite[Prop.\ 5.3.7]{FGauges}.
\end{proof}

\begin{lem}
\label{lem:fildhod-complexeszpntzero}
There is an equivalence of categories
\begin{equation*}
\D((\Z_p^\N)_{t=0})\cong \widehat{\D}_{\gr, D-\nilp}(\Z_p\{u, D\}/(Du-uD-1))\;,
\end{equation*}
where $u$ and $D$ have grading degree \EEDIT{$1$} and \EEDIT{$-1$}, respectively, and the \EDIT{local} nilpotence of $D$ is only demanded mod $p$ here.
\end{lem}
\begin{proof}[Proof]
Using \cref{prop:fildhod-zpntzero} to identify $(\Z_p^\N)_{t=0}$ as $\G_{a, +}^\dR/\G_m$, the statement is completely analogous to the result of \cite[§6.4.5]{FGauges}, which is in the mod $p$ setting. Nevertheless, we briefly \EDIT{lay out} the argument and first show that
\begin{equation}
\label{eq:fildhod-complexeszpntzero}
\D(\G_a^\dR)=\D(\G_a/\G_a^\sharp)\cong \D_{D-\nilp}(\Z_p\{u, D\}/(Du-uD-1))\;;
\end{equation}
then the claim will follow as specifying a $\G_m$-equivariant structure \EEDIT{on a quasi-coherent complex on $\G_a^\dR$} then amounts to specifying a grading \EEDIT{on its image under the equivalence above which is} compatible with the \EEDIT{grading on} $\Z_p\{u, D\}/(Du-uD-1)$.

To prove (\ref{eq:fildhod-complexeszpntzero}), first recall that the Cartier dual of $\G_a^\sharp$ is the formal completion $\widehat{\G}_a$ of $\G_a$ at the origin, see \cite[Prop.\ 2.4.4]{FGauges}, and that an endomorphism $D: M\rightarrow M$ of a $p$-complete $\Z_p$-module $M$ which is locally nilpotent mod $p$ corresponds to the $\O(\G_a^\sharp)$-comodule structure on $M$ given by
\begin{equation*}
M\rightarrow M\widehat{\tensor}_{\Z_p} \O(\G_a^\sharp)\;, \hspace{0.5cm} m\mapsto \sum_{i\geq 0} D^i(m)\tensor\frac{x^i}{i!}\;,
\end{equation*}
see \cite[Prop.\ 2.4.4]{FGauges}; here, $\O(\G_a^\sharp)$ denotes the $p$-adic completion of $\Z_p[x, \tfrac{x^2}{2!}, \tfrac{x^3}{3!}, \dots]$. Thus, equipping a $p$-complete $\Z_p\langle u\rangle$-module $M$ with an $\O(\G_a^\sharp)$-equivariant structure means specifiying an endomorphism $D: M\rightarrow M$ which is locally nilpotent mod $p$ such that the diagram
\begin{equation*}
\begin{tikzcd}
M\widehat{\tensor}_{\Z_p} \Z_p\langle u\rangle\ar[r]\ar[d] & M\ar[r] & M\widehat{\tensor}_{\Z_p} \O(\G_a^\sharp) \\
(M\widehat{\tensor}_{\Z_p} \O(\G_a^\sharp))\widehat{\tensor}_{\Z_p} (\Z_p\langle u\rangle\widehat{\tensor}_{\Z_p} \O(\G_a^\sharp))\ar[rr, "\cong"] && (M\widehat{\tensor}_{\Z_p} \Z_p\langle u\rangle)\widehat{\tensor}_{\Z_p} (\O(\G_a^\sharp)\widehat{\tensor}_{\Z_p} \O(\G_a^\sharp))\ar[u]
\end{tikzcd}
\end{equation*}
commutes. Now observing that the endomorphism of $\Z_p\langle u\rangle$ corresponding to the action of $\G_a^\sharp$ on $\G_a$ is given by the usual derivative $\frac{\d}{\d u}: \Z_p\langle u\rangle\rightarrow\Z_p\langle u\rangle$, we see that the condition above translates to
\begin{equation*}
D^i(fm)=\sum_{k+\ell=i} \binom{i}{k}\left(\frac{\d^\ell}{\d^\ell u} f\right) D^k(m)
\end{equation*}
for any $i\geq 0$ and $f\in\Z_p\langle u\rangle, m\in M$. \EDIT{As this is clearly an additive condition, it suffices to demand this for monomials $f=u^n$, where $n\geq 0$. However, a straightforward induction on $n$ then shows} that the relation $D(um)=uD(m)+m$ already implies all the others, hence specifying a $\G_a^\sharp$-equivariant structure on $M$ amounts to specifying a lift of $M$ to a $\Z_p\{u, D\}/(Du-uD-1)$-module such that $D$ acts locally nilpotently on $M$ mod $p$ and we are done.
\end{proof}

\begin{rem}
\label{rem:fildhod-complexeszpntzeroexplicit}
In light of \cref{lem:fildhod-complexeszpntzero}, the datum of a quasi-coherent complex $E$ on $(\Z_p^\N)_{t=0}$ may be thought of as follows: Via the Rees equivalence, $E$ corresponds to an increasing filtration $\Fil_\bullet V$ of $p$-complete $\Z_p$-complexes equipped with an endomorphism $D: \Fil_\bullet V\rightarrow\Fil_{\bullet-1} V$ such that the diagram
\begin{equation}
\label{eq:fildhod-complexeszpntzeroexplicit}
\begin{tikzcd}
\dots\ar[r] & \Fil_{-1} V\ar[r, "u"]\ar[d, "uD+1"] & \Fil_0 V\ar[r, "u"]\ar[d, "uD"] & \Fil_1 V\ar[r, "u"]\ar[d, "uD-1"] & \dots\ar[r, "u"] & \Fil_i V\ar[r]\ar[d, "uD-i"] & \dots \\
\dots\ar[r] & \Fil_{-1} V\ar[r, "u"] & \Fil_0 V\ar[r, "u"] & \Fil_1 V\ar[r, "u"] & \dots\ar[r, "u"] & \Fil_i V\ar[r] & \dots
\end{tikzcd}
\end{equation}
commutes. Note that the filtration $\Fil_\bullet V$ precisely corresponds to the pullback of $E$ to $\A^1_+/\G_m$ under the Rees equivalence.
\end{rem}

The Sen operator on the conjugate-filtered diffracted Hodge cohomology of $X$ is captured by the stack $(X^\N)_{t=0}$ in the following way:

\begin{prop}
\label{prop:fildhod-comparisonsen}
Let $X$ be a bounded $p$-adic formal scheme and assume that $X$ is $p$-quasisyntomic and qcqs. Under the equivalence
\begin{equation*}
\D((\Z_p^\N)_{t=0})\cong \widehat{\D}_{\gr, D-\nilp}(\Z_p\{u, D\}/(Du-uD-1))
\end{equation*}
from \cref{lem:fildhod-complexeszpntzero}, the underlying graded \EEDIT{$\Z_p\langle u\rangle$-complex} of $\pi_{X^\N, *}\O_{(X^\N)_{t=0}}$ identifies \EDIT{with} $\Fil_\bullet^\conj R\Gamma_\dHod(X)$ under the Rees equivalence \EEDIT{and, under this identification, the operator 
\begin{equation*}
uD-i: \Fil_i^\conj R\Gamma_\dHod(X)\rightarrow \Fil_i^\conj R\Gamma_\dHod(X)
\end{equation*}
identifies with the Sen operator on $\Fil_i^\conj R\Gamma_\dHod(X)$ for all $i\in\Z$.}
\end{prop}
\begin{proof}
By \cref{thm:fildhod-comparisonfiltered}, we already know that the underlying filtered objects agree, so it remains to \EEDIT{prove the assertion about the Sen operator. However, as the conjugate filtration on the diffracted Hodge cohomology of $X$ is complete and the diagram (\ref{eq:fildhod-complexeszpntzeroexplicit}) is commutative, we may reduce to checking the corresponding statement on the associated graded level by induction. Here, we observe that 
\begin{equation*}
uD-i: \Fil_i^\conj R\Gamma_\dHod(X)\rightarrow \Fil_i^\conj R\Gamma_\dHod(X)
\end{equation*}
acts by $-i$ on $\gr_i^\conj R\Gamma_\dHod(X)$ since $uD$ factors through $\Fil_{i-1}^\conj R\Gamma_\dHod(X)$. Recalling that we deduced from (\ref{eq:fildhod-htcomparison1}) that the Sen operator also acts by $-i$ on $\gr_i^\conj R\Gamma_\dHod(X)$, we obtain the result.}
\end{proof}

Finally, we note that, as promised in the beginning of this section, by \cref{rem:fildhod-complexeszpntzeroexplicit}, the $\Z_p\{u, D\}/(uD-Du-1)$-module structure on conjugate-filtered diffracted Hodge cohomology supplied by \cref{prop:fildhod-comparisonsen} not only encodes the Sen operator $\Theta$, but also the fact that, for all $n\in\Z$, the endomorphism
\begin{equation*}
\Theta+n: \Fil^\conj_n R\Gamma_\dHod(X)\rightarrow \Fil^\conj_n R\Gamma_\dHod(X)
\end{equation*}
admits a factorisation through $\Fil^\conj_{n-1} R\Gamma_\dHod(X)$ as remarked in \cite[Rem.\ 4.9.10]{APC}. \EEDIT{This is due to the fact that $\Theta+n$ identifies with $(uD-n)+n=uD$, which factors canonically through $u: \Fil^\conj_{n-1} R\Gamma_\dHod(X)\rightarrow \Fil^\conj_n R\Gamma_\dHod(X)$ via $D$.} 

\section{Proof of the main result}
\label{sect:nygaardhodge}

In this section, we prove the announced stacky version of the comparison theorem \cref{thm:nygaardhodge-motivation} of Bhatt--Lurie between the Nygaard filtration on prismatic cohomology and the Hodge filtration on de Rham cohomology and derive some corollaries. In fact, we will deduce \cref{thm:nygaardhodge-main} from the following statement, which, in some sense, is the analogous assertion on the level of the associated graded:

\begin{prop}
\label{prop:nygaardhodge-graded}
The Hodge-filtered de Rham map 
\begin{equation*}
i_{\dR, +}: \Z_p^{\dR, +}\rightarrow\Z_p^\N
\end{equation*}
is an almost isomorphism over the locus $t=0$ up to $p$-isogeny. \EDIT{More precisely,} for any $E\in\Perf(\Z_p^\N)$, it induces an isomorphism
\begin{equation*}
R\Gamma((\Z_p^\N)_{t=0}, E|_{(\Z_p^\N)_{t=0}})[\tfrac{1}{p}]\cong R\Gamma(\Z_p^\Hod, i_\Hod^*E)[\tfrac{1}{p}]\;.
\end{equation*}
If the Hodge--Tate weights of $E$ are all at least \EEDIT{$-(p-1)$}, then the statement already holds integrally.
\end{prop} 

Before we can give the proof of the above statement, we need to prepare ourselves with two easy lemmas:

\begin{lem}
\label{lem:nygaardhodge-rgammazpntzero}
Let $E\in\Perf(\Z_p^\N)$ and identify its pushforward to $\A^1_-/\G_m$ \EEDIT{along} the Rees map with a filtered complex
\begin{equation*}
\begin{tikzcd}
\dots\ar[r] & M^{i+1}\ar[r] & M^i\ar[r] & M^{i-1} \ar[r] & \dots
\end{tikzcd}
\end{equation*}
Then we have
\begin{equation*}
R\Gamma((\Z_p^\N)_{t=0}, E)=M^0/M^1\;.
\end{equation*}
\end{lem}
\begin{proof}
\EDIT{As pullback along $B\G_m\rightarrow\A^1_-/\G_m$ corresponds to passage to the associated graded under the Rees equivalence}, we have 
\begin{equation*}
M^0/M^1\cong R\Gamma(B\G_m, (t_* E)|_{B\G_m})\;,
\end{equation*}
so we have to prove that the cartesian diagram
\begin{equation*}
\begin{tikzcd}
(\Z_p^\N)_{t=0}\ar[r]\ar[d, "t", swap] & \Z_p^\N\ar[d, "t"] \\
B\G_m\ar[r] & \A^1_-/\G_m
\end{tikzcd}
\end{equation*}
satisfies base change for perfect complexes, i.e.\ that
\begin{equation*}
(t_*E)|_{B\G_m}\cong t_*E|_{(\Z_p^\N)_{t=0}}
\end{equation*}
via the natural map. However, this follows from the argument given in the proof of \cite[Prop.\ 3.3.5]{FGauges}.
\end{proof}

\begin{cor}
\label{cor:nygaardhodge-fibreseq}
In the situation of \cref{lem:nygaardhodge-rgammazpntzero}, identify $E|_{(\Z_p^\N)_{t=0}}$ with a filtered object $\Fil_\bullet V$ together with an operator $D$ as in \cref{rem:fildhod-complexeszpntzeroexplicit}. Then there is a natural fibre sequence
\begin{equation*}
\begin{tikzcd}
M^0/M^1\ar[r] & \Fil_0 V\ar[r, "D"] & \Fil_{-1} V\;.
\end{tikzcd}
\end{equation*}
\end{cor}
\begin{proof}
By \cref{lem:fildhod-complexeszpntzero}, we have
\begin{equation*}
R\Gamma((\Z_p^\N)_{t=0}, E|_{(\Z_p^\N)_{t=0}})\cong\fib(\Fil_0 V\xrightarrow{D}\Fil_{-1} V)\;.
\end{equation*}
Combining this with \cref{lem:nygaardhodge-rgammazpntzero} yields the claim.
\end{proof}

\begin{rem}
In light of the comparisons from Chapter \ref{sect:conjdhod}, one should view the above statement as a version of \cite[Rem.\ 5.5.8]{APC} with coefficients. 
\end{rem}

We are now ready to prove \cref{thm:nygaardhodge-main} and \cref{prop:nygaardhodge-graded}:

\begin{proof}[Proof of \cref{prop:nygaardhodge-graded}]
Consider the commutative diagram
\begin{equation}
\label{eq:nygaardhodge-graded}
\begin{tikzcd}
\Z_p^\dR\ar[r]\ar[d, "i_\dR", swap] & \Z_p^{\dR, +}\ar[d, "i_{\dR, +}"] & \Z_p^\Hod\ar[d, "i_\Hod"]\ar[l] \\
\Z_p^\prism\ar[r, "j_\dR"]\ar[d] & \Z_p^\N\ar[d, "t"] & (\Z_p^\N)_{t=0}\ar[d]\ar[l] \\
\Spf\Z_p\ar[r] & \A^1_-/\G_m & B\G_m\ar[l]\nospacepunct{\;,}
\end{tikzcd}
\end{equation}
where the vertical compositions are all identities since the Hodge-filtered de Rham map $i_{\dR, +}$ is a section of the Rees map $t$. We identify $t_*E$ with a filtered complex
\begin{equation*}
\begin{tikzcd}
\dots\ar[r] & M^{i+1}\ar[r] & M^i\ar[r] & M^{i-1}\ar[r] & \dots
\end{tikzcd}
\end{equation*}
via the Rees equivalence and use \cref{lem:fildhod-complexeszpntzero} to identify $E|_{(\Z_p^\N)_{t=0}}$ with a filtered complex $\Fil_\bullet V$ together with an operator $D: \Fil_\bullet V\rightarrow\Fil_{\bullet-1} V$ satisfying the \EEDIT{compatibility} properties discussed \EEDIT{in \cref{rem:fildhod-complexeszpntzeroexplicit}}; finally, we let $i_{\dR, +}^* E$ correspond to the filtered complex $\Fil^\bullet F$ under the Rees equivalence. Using \cref{lem:nygaardhodge-rgammazpntzero}, we see that
\begin{equation*}
R\Gamma((\Z_p^\N)_{t=0}, E|_{(\Z_p^\N)_{t=0}})\cong M^0/M^1\;, \hspace{0.5cm} R\Gamma(\Z_p^\Hod, i_\Hod^*E)\cong \Fil^0 F/\Fil^1 F
\end{equation*}
and thus, we have to examine the fibre of the map $M^0/M^1\rightarrow \Fil^0 F/\Fil^1 F$ induced by the Hodge-filtered de Rham map $i_{\dR, +}$. To this end, first observe that the commutativity of the top right square of (\ref{eq:nygaardhodge-graded}) implies that there is an isomorphism $\Fil^0 F/\Fil^1 F\cong \Fil_0 V/\Fil_{-1} V$. Thus, using \cref{cor:nygaardhodge-fibreseq}, we obtain a commutative diagram of fibre sequences
\begin{equation*}
\begin{tikzcd}
& M^0/M^1\ar[r]\ar[d] & \Fil^0 F/\Fil^1 F\ar[d, "\cong"] \\
\Fil_{-1} V\ar[r, "u"]\ar[d, "Du", swap] & \Fil_0 V\ar[d, "D"]\ar[r] & \Fil_0 V/\Fil_{-1} V\ar[d] \\
\Fil_{-1} V\ar[r, equals] & \Fil_{-1} V\ar[r] & 0\nospacepunct{\;,}
\end{tikzcd}
\end{equation*}
from which we conclude that
\begin{equation*}
\fib(M^0/M^1\rightarrow \Fil^0 F/\Fil^1 F)\cong\fib(\Fil_{-1} V\xrightarrow{Du} \Fil_{-1} V)\cong \fib(\Fil_{-1} V\xrightarrow{uD+1} \Fil_{-1} V)\;,
\end{equation*}
\EEDIT{where the last step is due to $Du=uD+1$.}
As $E$ is perfect, we have $\Fil_i V=0$ for $i\ll 0$ and thus, \EEDIT{by the commutative diagram (\ref{eq:fildhod-complexeszpntzeroexplicit}),} we are reduced to showing that \EEDIT{the operators}
\begin{equation*}
uD-i: \Fil_i V\rightarrow \Fil_i V
\end{equation*}
induce $p$-isogenies on $\gr_i V$ for all $i\leq -1$ and isomorphisms for \EEDIT{$-(p-1)\leq i\leq -1$}. However, this is clear since \EEDIT{$uD-i$} acts by multiplication by \EEDIT{$-i$} on $\gr_i V$.
\end{proof}

\begin{proof}[Proof of \cref{thm:nygaardhodge-main}]
We use the notations from the previous proof. As
\begin{equation*}
R\Gamma(\Z_p^\N, E)\cong M^0\;, \hspace{0.5cm} R\Gamma(\Z_p^\prism, j_\dR^*E)\cong M^{-\infty}\;,
\end{equation*}
where $M^{-\infty}$ denotes the $p$-completion of $\colim_i M^{-i}$, and
\begin{equation*}
R\Gamma(\Z_p^{\dR, +}, i_{\dR, +}^*E)\cong \Fil^0 F\;, \hspace{0.5cm} R\Gamma(\Z_p^\dR, i_\dR^*E)\cong F\;,
\end{equation*}
where $F$ denotes the $p$-completion of $\colim_i \Fil^{-i} F$, we have to examine the map $F/\Fil^0 F\rightarrow M^{-\infty}/M^0$ induced by $i_{\dR, +}$. Since $E$ is perfect, the filtrations $\Fil^\bullet F$ and $M^\bullet$ eventually stabilise, i.e.\ we have $F=\Fil^i F$ and $M^{-\infty}=M^i$ for $i\ll 0$, and thus we are reduced to showing that the maps $\gr^i F\rightarrow M^i/M^{i+1}$ are $p$-isogenies for all $i\leq -1$ and isomorphisms if $E$ has Hodge--Tate weights all at least \EEDIT{$-p$}. However, note that, using \cref{lem:nygaardhodge-rgammazpntzero}, we have
\begin{equation*}
\EEDIT{
\begin{split}
R\Gamma((\Z_p^\N)_{t=0}, (E\tensor t^*\O(i))|_{(\Z_p^\N)_{t=0}})&\cong M^i/M^{i+1} \\
R\Gamma(\Z_p^\Hod, (i_{\dR, +}^*(E\tensor t^*\O(i)))|_{\Z_p^\Hod})&\cong \gr^i F
\end{split}
}
\end{equation*}
and if $E$ has Hodge--Tate weights all at least \EEDIT{$-p$}, then \EEDIT{$E\tensor t^*\O(i)$} has Hodge--Tate weights all at least \EEDIT{$-(p+i)$}, so we are done by \cref{prop:nygaardhodge-graded}.
\end{proof}

Finally, we derive some easy corollaries of \cref{thm:nygaardhodge-main}, the first of which is \cref{cor:nygaardhodge-coeffs} already stated in the introduction.

\begin{proof}[Proof of \cref{cor:nygaardhodge-coeffs}]
Using the fact that pushforward along $X^\N\rightarrow\Z_p^\N$ preserves perfect complexes if $X$ is smooth and proper, see \cite[Prop.\ B.0.1]{EtaleCrystalline}, the statement is an immediate consequence of \cref{thm:nygaardhodge-main} once we know that pushforward along $X^\N\rightarrow\Z_p^\N$ does not decrease the smallest Hodge--Tate weight of a gauge; \EDIT{however, this follows from} \cref{rem:fildrstack-hodstackcoh}.
\end{proof}

Finally, observe that \cref{thm:nygaardhodge-main} and \cref{prop:nygaardhodge-graded} also apply to certain non-perfect gauges $E$: Namely, going through the proof, one sees that all we have used about $E$ is that it is bounded below with respect to the standard $t$-structure, that the associated decreasing filtrations $\Fil^\bullet F$ and $M^\bullet$ eventually stabilise and that the terms of the increasing filtration $\Fil_\bullet V$ vanish in sufficiently negative degrees. We can use this observation to recover \cref{thm:nygaardhodge-motivation} from the beginning:

\begin{proof}[Proof of \cref{thm:nygaardhodge-motivation}]
Consider the gauge \EEDIT{$E=\H_\N(X)(i)$}, where the twist is pulled back from $\A^1_-/\G_m$ via the Rees map $t: \Z_p^\N\rightarrow\A^1_-/\G_m$. Noting that $i_{\dR, +}^*E, j_\dR^*E$ and $i_\dR^*E$ identify with \EEDIT{$\H_{\dR, +}(X)(i), \H_\prism(X)$} and $\H_\dR(X)$, respectively, we see that, by the comparisons from \cref{thm:fildrstack-comparison}, \cref{thm:prismatisation-comparison} and \cref{thm:filprism-comparison}, it suffices to show that \cref{thm:nygaardhodge-main} also applies to $E$. However, note that these comparison results also imply that the filtrations $\Fil^\bullet F$ and $M^\bullet$ identify (up to a shift of degrees) with the Hodge and Nygaard filtration, respectively, so they indeed eventually stabilise. Finally, to prove that the terms of $\Fil_\bullet V$ vanish in sufficiently negative degrees, we observe that this filtration identifies (up to a shift of degrees) with $\Fil_\bullet^\conj\Omega_X^\dHod$ by \cref{prop:fildhod-comparisonsen}, whose negative terms are all zero. As $E$ is clearly bounded below, we are done.
\end{proof}

\appendix
\section{The conjugate-filtered Hodge--Tate stack}

In this appendix, we describe a stack $X^{\HT, c}$ attached to any $p$-adic formal scheme $X$ which is equipped with a map to $\A^1_+/\G_m$ such that the pushforward of $\O_{X^{\HT, c}}$ to $\A^1_+/\G_m$ identifies with the (absolute) conjugate-filtered Hodge--Tate cohomology of $X$ under the Rees equivalence.

\comment{recall from \cite[Constr.\ 3.6.1]{APC} that there is a ``Frobenius'' $F: \Z_p^\prism\rightarrow\Z_p^\prism$ given by taking a Cartier--Witt divisor $I\xrightarrow{\alpha} W(S)$ and pulling it back along the Frobenius endomorphism of $W(S)$. Moreover, this induces an endomorphism $F: X^\prism\rightarrow X^\prism$, see \cite[Rem.\ 3.6]{PFS}. Also recall that the Nygaard-filtered prismatisation $X^\N$ is equipped with a map $\pi_X: X^\N\rightarrow X^\prism$ called the \emph{structure map}, which is given by sending a filtered Cartier--Witt divisor $M\xrightarrow{d} W$ with associated map of admissible sequences
\begin{equation*}
\begin{tikzcd}
0\ar[r] & \V(L_M)^\sharp\ar[d, "\sharp(d)"]\ar[r] & M\ar[d, "d"]\ar[r] & F_*M'\ar[d, "F_*(d')"]\ar[r] & 0 \\
0\ar[r] & \G_a^\sharp\ar[r] & W\ar[r] & F_*W\ar[r] & 0
\end{tikzcd}
\end{equation*}
to the Cartier--Witt divisor giving rise to the map $d': M'\rightarrow W$.

\begin{defi}
The \emph{conjugate-filtered Hodge--Tate stack} of $X$ is the stack $X^{\HT, c}$ defined by the pullback diagram
\begin{equation*}
\begin{tikzcd}
X^{\HT, c}\ar[r]\ar[d] & (X^\N)_{t=0}\ar[d, "\pi_X"] \\
X^\prism\ar[r, "F"] & X^\prism\nospacepunct{\;.}
\end{tikzcd}
\end{equation*}
\end{defi}

We first explain how to obtain the desired map $X^{\HT, c}\rightarrow\A^1_+/\G_m$ and for this,}

To define $X^{\HT, c}$, it will be easiest to invoke the alternative description of $\Z_p^\N$ recently given by Gardner--Madapusi in \cite[§6.4]{GardnerMadapusi}. Namely, they show that one may also obtain $\Z_p^\N$ as the pullback
\begin{equation*}
\begin{tikzcd}
\Z_p^\N\ar[r, "\pi"]\ar[d, "{(t, u)}", swap] & \Z_p^\prism\ar[d] \\
\A^1_-/\G_m\times (\A^1_+/\G_m)^\dR\ar[r] & (\A^1_-/\G_m)^\dR\nospacepunct{\;,}
\end{tikzcd}
\end{equation*}
where the two nonlabelled maps are given as follows:
\begin{enumerate}[label=(\roman*)]
\item The bottom map sends a pair $(L_-\rightarrow S, \G_a^\dR(S)\rightarrow L_+)$ of maps of invertible $S$-modules or $\G_a^\dR(S)$-modules, respectively, to the map
\begin{equation*}
L_-\tensor_S L_+^{-1}\rightarrow S\tensor_S \G_a^\dR(S)\cong \G_a^\dR(S)
\end{equation*}
of invertible $\G_a^\dR(S)$-modules.

\item The map on the right sends a Cartier--Witt divisor $I\rightarrow W(S)$ to the generalised Cartier divisor $F_*I\tensor_{F_*W(S)} \G_a^\dR(S)\rightarrow \G_a^\dR(S)$ on $\G_a^\dR(S)$ given by pulling back $F_*I$ along the natural map $F_*W(S)\rightarrow (F_*W/^\mathbb{L} p)(S)\cong\G_a^\dR(S)$, see \cite[Cor.\ 2.6.8]{FGauges} for the last isomorphism.
\end{enumerate}

While the maps $t$ and $\pi$ are the familiar Rees and structure maps, respectively, the map $u$ is not as easy to construct from the moduli interpretation $\Z_p^\N$ using filtered Cartier--Witt divisors. The rough idea is as follows: fpqc-locally, any admissible $W$-module $M$ arises from the standard exact sequence \cref{lem:admissible-fundamentalseq} via pushout along an endomorphism of $\G_a^\sharp$, see \cite[Lem.\ 5.2.8]{FGauges}; moreover, $\sHom_W(\G_a^\sharp, \G_a^\sharp)\cong\G_a$ and two elements of $\G_a$ yield the same admissible $W$-module $M$ if and only if they differ by an element of $\G_a^\sharp$, see \cite[Prop.\ 5.2.1]{FGauges}. Thus, fpqc-locally, any admissible $W$-module $M$ gives rise to a section of $\G_a^\dR$ and this defines the map $u$.

\comment{
\begin{lem}
\label{lem:conjHT-jHT}
There is a pullback square
\begin{equation*}
\begin{tikzcd}
\Z_p^\prism\ar[r, "F"]\ar[d] & \Z_p^\prism\ar[d] \\
\A^1_-/\G_m\ar[r] & (\A^1_-/\G_m)^\dR\nospacepunct{\;.}
\end{tikzcd}
\end{equation*}
\end{lem}
\begin{proof}
This is just a reformulation of the fact that there is an open immersion $j_\HT: \Z_p^\prism\hookrightarrow\Z_p^\N$ living over $F: \Z_p^\prism\rightarrow\Z_p^\prism$, see \cite[Constr.\ 5.3.2]{FGauges}, using the description of $j_\HT$ in the framework of Gardner--Madapusi, see \cite[Def.\ 6.5.2]{GardnerMadapusi}. Alternatively, the statement can also be obtained from the pullback square 
\begin{equation*}
\begin{tikzcd}
W\ar[r, "F"]\ar[d] & F_*W\ar[d] \\
\G_a\ar[r] & \G_a^\dR
\end{tikzcd}
\end{equation*}
induced by the exact sequence from \cref{lem:admissible-fundamentalseq}.
\end{proof}
}

With this description of $\Z_p^\N$ in place, we can define the stack $X^{\HT, c}$ as follows:

\begin{defi}
The \emph{conjugate-filtered Hodge--Tate stack} of $X$ is the stack $X^{\HT, c}$ defined by the pullback diagram
\begin{equation*}
\begin{tikzcd}
X^{\HT, c}\ar[r]\ar[d] & (X^\N)_{t=0}\ar[d, "u_X"] \\
\A^1_+/\G_m\ar[r] & (\A^1_+/\G_m)^\dR\nospacepunct{\;,}
\end{tikzcd}
\end{equation*}
where $u_X$ denotes the composition $(X^\N)_{t=0}\rightarrow (\Z_p^\N)_{t=0}\xrightarrow{u} (\A^1_+/\G_m)^\dR$.
\end{defi}

\begin{prop}
\label{prop:conjHT-zphtc}
We have 
\begin{equation*}
\Z_p^{\HT, c}\cong \A^1_+/\G_m\times B\G_m^\sharp\;.
\end{equation*}
\end{prop}
\begin{proof}
By definition, the stack $\Z_p^{\HT, c}$ fits into a pullback diagram
\begin{equation*}
\begin{tikzcd}
\Z_p^{\HT, c}\ar[r]\ar[d] & (\Z_p^\N)_{t=0}\ar[d, "u"] \\
\A^1_+/\G_m\ar[r] & (\A^1_+/\G_m)^\dR\nospacepunct{\;.}
\end{tikzcd}
\end{equation*}
As $\G_m^\dR\cong \G_m/\G_m^\sharp$ and taking de Rham stacks commutes with quotients, the map on the right is a $\G_m^\sharp$-gerbe by \cref{prop:fildhod-zpntzero} and hence the same is true for the map on the left. However, the gerbe on the left is split: using \cref{prop:fildhod-zpntzero} once again, the splitting is induced by the natural map $\A^1_+/\G_m\rightarrow \G_{a, +}^\dR/\G_m\cong (\Z_p^\N)_{t=0}$.
\end{proof}

\comment{
Now the desired map $X^{\HT, c}\rightarrow\A^1_+/\G_m$ is obtained as follows: by functoriality of the construction, there is a map $X^{\HT, c}\rightarrow\Z_p^{\HT, c}$ and then we postcompose this with the natural projection $\Z_p^{\HT, c}\cong \A^1_+/\G_m\times B\G_m^\sharp\rightarrow\A^1_+/\G_m$.
}

\begin{cor}
\label{cor:conjHT-complexeszphtc}
There is an equivalence of categories
\begin{equation*}
\D(\Z_p^{\HT, c})\cong \widehat{\DF}_{(\Theta^p-\Theta)-\nilp}(\Z_p[\Theta])\;,
\end{equation*}
where the right-hand side denotes the full subcategory of $\widehat{\DF}(\Z_p[\Theta])$ consisting of objects on which $\Theta^p-\Theta$ acts locally nilpotently mod $p$. Moreover, given $E\in\D(\Z_p^{\HT, c})$ which identifies with $\Fil_\bullet M\in\widehat{\DF}_{(\Theta^p-\Theta)-\nilp}(\Z_p[\Theta])$, the pushforward of $E$ to $\A^1_+/\G_m$ identifies with $(\Fil_\bullet M)^{\Theta=0}$ under the Rees equivalence (here, the kernel is to be taken in the derived sense).
\end{cor}
\begin{proof}
Using \cref{prop:conjHT-zphtc}, this is a variant of \cref{prop:fildhod-bgmsharp} and \cite[Prop.\ 3.5.11]{APC}.
\end{proof}

\begin{prop}
\label{prop:conjHT-main}
Let $X$ be a bounded $p$-adic formal schemes and assume that $X$ is $p$-quasisyntomic and qcqs. Then the pushforward of $\O_{X^{\HT, c}}$ along $\pi_{X^{\HT, c}}: X^{\HT, c}\rightarrow\Z_p^{\HT, c}$ identifies with the conjugate-filtered diffracted Hodge cohomology equipped with its Sen operator under the equivalence from \cref{cor:conjHT-complexeszphtc}.
\end{prop}
\begin{proof}
Using base change for the cartesian square
\begin{equation*}
\begin{tikzcd}
X^{\HT, c}\ar[r]\ar[d] & (X^\N)_{t=0}\ar[d] \\
\Z_p^{\HT, c}\ar[r] & (\Z_p^\N)_{t=0}\nospacepunct{\;,}
\end{tikzcd}
\end{equation*}
we know that $\pi_{X^{\HT, c}, *}\O_{X^{\HT, c}}$ is the pullback of $\pi_{X^\N, *}\O_{(X^\N)_{t=0}}$ along the map 
\begin{equation*}
\Z_p^{\HT, c}\cong \A^1_+/\G_m\times B\G_m^\sharp\rightarrow \G_{a, +}^\dR/\G_m\cong (\Z_p^\N)_{t=0}\;.
\end{equation*}
Moreover, recall from \cref{prop:fildhod-comparisonsen} that $\pi_{X^\N, *}\O_{(X^\N)_{t=0}}$ identifies with $\Fil^\conj_\bullet R\Gamma_\dHod(X)$ equipped with operators $D: \Fil^\conj_\bullet R\Gamma_\dHod(X)\rightarrow\Fil^\conj_{\bullet-1} R\Gamma_\dHod(X)$ such that $uD-i: \Fil^\conj_i R\Gamma_\dHod(X)\rightarrow\Fil^\conj_i R\Gamma_\dHod(X)$ agrees with the Sen operator, where $u: \Fil^\conj_\bullet R\Gamma_\dHod(X)\rightarrow\Fil^\conj_{\bullet+1} R\Gamma_\dHod(X)$ denotes the transition maps.

Using the commutative diagram
\begin{equation*}
\begin{tikzcd}
& \A^1_+/\G_m\ar[rd]\ar[ld] & \\
\A^1_+/\G_m\times B\G_m^\sharp\ar[rr] & & \G_{a, +}^\dR/\G_m\nospacepunct{\;,}
\end{tikzcd}
\end{equation*}
where all maps are the canonical ones, we thus first conclude that the pullback of $\pi_{X^{\HT, c}, *}\O_{X^{\HT, c}}$ to $\A^1_+/\G_m$ identifies with $\Fil^\conj_\bullet R\Gamma_\dHod(X)$ under the Rees equivalence. It thus only remains to identify the operator $\Theta$ coming from \cref{cor:conjHT-complexeszphtc} with the Sen operator. To do this, we may reduce to the associated graded using the fact that the filtration $\Fil^\conj_\bullet R\Gamma_\dHod(X)$ is complete. By the previous paragraph, we thus have to check that $\Theta$ acts by $-i$ on $\gr^\conj_i R\Gamma_\dHod(X)$.

However, for this, note that the composition
\begin{equation*}
B\G_m\times B\G_m^\sharp\rightarrow \A^1_+/\G_m\times B\G_m^\sharp\rightarrow \G_{a, +}^\dR/\G_m
\end{equation*}
may alternatively be factored as
\begin{equation*}
B\G_m\times B\G_m^\sharp\rightarrow B\G_m\times B\G_m\rightarrow B\G_m\rightarrow \G_{a, +}^\dR/\G_m\nospacepunct{\;,}
\end{equation*}
where the middle arrow takes a pair $(L, L')$ of invertible $S$-modules and sends it to $L\tensor_S L'$; in other words, the pullback of $\O(1)$ along the middle arrow is given by $\O(1)\boxtimes\O(1)$. As the pullback of $\pi_{X^\N, *}\O_{(X^\N)_{t=0}}$ to $B\G_m$ identifies with $\gr^\conj_\bullet R\Gamma_\dHod(X)$ by what we already know and the pullback of $\O(i)$ along $B\G_m^\sharp\rightarrow B\G_m$ corresponds to $\Z_p$ equipped with the multiplication-by-$i$-map under the equivalence from \cref{prop:fildhod-bgmsharp}, we are done.
\end{proof}

\begin{cor}
The pushforward of $\O_{X^{\HT, c}}$ to $\A^1_+/\G_m$ identifies with the conjugate-filtered Hodge--Tate cohomology $\Fil_\bullet^\conj R\Gamma_{\ol{\prism}}(X)$ of $X$ under the Rees equivalence.
\end{cor}
\begin{proof}
This now follows by combining \cref{cor:conjHT-complexeszphtc} and \cref{prop:conjHT-main} with the fact that
\begin{equation*}
\Fil_\bullet^\conj R\Gamma_{\ol{\prism}}(X)\cong (\Fil_\bullet^\conj R\Gamma_\dHod(X))^{\Theta=0}\;,
\end{equation*}
see \cite[Rem.\ 4.7.5]{APC}.
\end{proof}

\bibliographystyle{amsalpha}
\bibliography{References}
\end{document}